\documentclass[12pt,reqno]{amsart}
\usepackage[innermargin=1.4in,outermargin=1.4in,bottom=1.5in,marginparwidth=1.1in,marginparsep=3mm]{geometry}
\usepackage{amsfonts}
 \usepackage{enumerate}

\usepackage{amsmath,amssymb,amsthm,bm}
\usepackage{mathtools}
\usepackage{shuffle}
\usepackage{enumitem}
\usepackage{amsrefs}
\usepackage{empheq}
\usepackage{wasysym}
 \usepackage{mathscinet}
\usepackage{verbatim}
\usepackage{graphicx}
\usepackage[
bookmarks=true,         
bookmarksnumbered=true, 
colorlinks=true, pdfstartview=FitV, linkcolor=blue, citecolor=blue,
urlcolor=blue]{hyperref}
\usepackage{microtype}
\usepackage{geometry}                
\usepackage{graphicx}
\usepackage{amssymb}
\usepackage{epstopdf}
\usepackage{amssymb}
\usepackage{amsmath}
\usepackage{amsfonts}
\usepackage{graphicx}
\usepackage{subfigure}
\usepackage{booktabs}
\usepackage{color}
\usepackage{array}
\usepackage{mathrsfs}
\usepackage{multirow}
\usepackage{float}
\usepackage{fancyhdr}

\theoremstyle{plain}
\newtheorem{theorem}{Theorem}[section]

\newtheorem{corollary}[theorem]{Corollary}

\newtheorem{lemma}[theorem]{Lemma}
\newtheorem{problem}[theorem]{Problem}
\newtheorem{proposition}[theorem]{Proposition}

\theoremstyle{remark}

\numberwithin{equation}{section}
\numberwithin{Assu}{section}
\numberwithin{Lem}{section}
\numberwithin{Defi}{section}
\numberwithin{Theo}{section}
\numberwithin{Pro}{section}
\numberwithin{Rem}{section}
\numberwithin{Coro}{section}
\numberwithin{Fig}{section}
\numberwithin{Exe}{section}

\theoremstyle{definition}

\newtheorem{remark}[theorem]{Remark} 

\newcommand{\EE}{\mathbb{E}}

\def\blue{\color{blue}}

\newcommand{\cF}{\mathcal{F}}

\newcommand{\cL}{\mathcal{L}}
\newcommand{\cM}{\mathcal{M}}

\newcommand{\cP}{\mathcal{P}}

\newcommand{\RR}{\mathbb{R} }

\newcommand{\al}{\alpha}

\newcommand{\ep}{\varepsilon}

\newcommand{\om}{\omega}

\newcommand{\si}{\sigma}

\let\Section=\section
\def\section{\setcounter{equation}{0}\Section}

\def\RR{\mathbb{R} }

\def\EE{\mathbb{E}}

\def\cF{{\mathcal F}}

\def\si{{\sigma}}

\def\cL{{\mathcal L}}

\def \RR{\mathbb R}

\def\EE{\mathbb E\ }

\def\cF{{\mathcal F}}

\def\si{{\sigma}}

\def\cL{{\mathcal L}}

\def\al{{\alpha}}

\def\si{{\sigma}}

\def\PP{{\mathbb P}}

\def\al{\alpha}

\def\Ad2{{\| A-\widetilde A \|_{L^{2}_{loc}} }}
\def\om{{\omega}}

\def\cP{{\mathcal P}}

\usepackage{amsfonts}
\usepackage{amsmath}
\usepackage{mathrsfs}
\usepackage{amssymb}
\usepackage{color,amscd,array,graphicx,mathdots,bm,amsfonts,epsfig,amsmath}
\usepackage{amsthm}
\usepackage{hyperref}
\usepackage{enumerate}
\usepackage{geometry}
\usepackage{yhmath}
\newcommand{\poincare}{Poincar\'{e}}

\makeatletter
\newcommand{\figcaption}{\def\@captype{figure}\caption}
\newcommand{\tabcaption}{\def\@captype{table}\caption}
\makeatother

\def\al{{\alpha}}
 
\def\om{{\omega}}\def\si{{\sigma}}

\newfam\msbmfam\font\tenmsbm=msbm10\textfont
\msbmfam=\tenmsbm\font\sevenmsbm=msbm7
\scriptfont\msbmfam=\sevenmsbm
\def\BB{\mathbb B}
\def\EE{\mathbb E}\def\PP{\mathbb P}
\def\RR{\mathbb R}
 
\def\<{\left<}\def\>{\right>}

\def\({\left(}\def\){\right)}

\numberwithin{equation}{section}

\begin{document}
\title[quantile-dependent McKean-Vlasov diffusions]
{Nonlinear McKean-Vlasov diffusions under the weak H\"{o}rmander condition with quantile-dependent coefficients}
\author{Yaozhong Hu}
\address{Department of Mathematical and Statistical Sciences, University of Alberta, Edmonton, AB T6G 2G1, Canada}
\email{yaozhong@ualberta.ca}
\thanks{Supported by NSERC Discovery grants and a startup fund from University of Alberta at Edmonton.}

\author{Michael A. Kouritzin}
\address{Department of Mathematical and Statistical Sciences, University of Alberta, Edmonton, AB T6G 2G1, Canada}
\email{michaelk@ualberta.ca}

\author{Jiayu Zheng}
\address{Department of Mathematical and Statistical Sciences, University of Alberta, Edmonton, AB T6G 2G1, Canada}
\email{jiayu8@ualberta.ca}
\thanks{Supported by National Natural Science Foundation of China grant 11901598. }

\date{}
\maketitle
\begin{abstract}
In this paper, the strong existence and  uniqueness for a degenerate finite system of quantile-dependent McKean-Vlasov stochastic differential equations are obtained under a weak H\"{o}rmander condition. The approach relies on the apriori    bounds for the density of the solution to time inhomogeneous diffusions. The time inhomogeneous Feynman-Fac formula is used to construct a contraction map for this degenerate system.
\bigskip
\end{abstract}


\section{Introduction}\label{s.1} 
Stochastic differential equations (SDEs)  with coefficients depending on the probability distribution of the unknown  have become 
a hot research area in recent years.  
One  particular topic is the  so-called mean field $d$-dimensional stochastic differential equations (see e.g. \cites{buckdahn, carmona1, carmona2} and references therein):
$
dX_t=F (t , \mathcal{L}(X_t),   X_t  ) dt+\sigma (t,    \mathcal{L}(X_t),   X_t )dW_t,
$
where $(W_t)_{t\geq 0}$ is a $d$-dimensional Brownian motion, and where $\mathcal{L}(X_t)\in \cP(\RR^d)$ is the probability law of
the unknown   $X_t$.   
To guarantee the existence and uniqueness of the solution, researchers often
assume that $F$ and $\si$ are  Lipschitzian  on $\cP(\RR^d) \times \RR^d$  with respect to  
a Wasserstein metric on the space $\cP(\RR^d)$ of probability measures.    
However, this condition is sometimes hard to verify. For example, in finance and other applications (e.g. \cite{crisan2014conditional}), the following quantile-dependent equation is  considered:  
\begin{align}\label{sde0}
dX_t=F (t, Q_{\alpha}(X_t),X_t)dt+\sigma (t, Q_{\alpha}(X_t),X_t)dW_t,
\end{align}
where  $F:\RR_{+}\times\RR^d\times \RR^d\to \RR^d$ and $\sigma : \RR_{+}\times\RR^d\times \RR^d\to \RR^d\otimes \RR^d$ are continuous functions, $\alpha=(\alpha_1,\dots,\alpha_d) \in (0, 1)^d$,  
and $Q_{\alpha}(X_t)$ is the $\alpha$-quantile (vector)  of the probability measure $\cL(X_t)$ of $X_t$, namely, 
\begin{align*}
\left(Q_{\alpha}(X_t)\right)_j=\left(Q_{\alpha}(\cL(X_t))\right)_j=\inf\Big\{y_j\in \RR,\int_{x\in\RR^d, x_j\leq y_j}\cL(X_t)(dx)\ge \al_j \Big\},\quad j=1,\dots, d\,. 
\end{align*} 
It is well-known that for any two real valued random variables
$X$ and $Y$ with cumulative distributions $F_X$ and $F_Y$, the 
p-Wasserstein  distance is given by
\[
I_p(X, Y)=\left(\int_0^1 |F_X^{-1}(\al)-F_Y^{-1}(\al)|^p d\al\right)^{1/p}\,.
\]
From the above expression it is obvious that the coefficients in \eqref{sde0} 
are not continuous with respect to the 
Wasserstein  distance for any finite $p\ge 1$.   Hence, we need a completely different approach to study the quantile-dependent equations. 
 
The works \cite{crisan2014conditional}  and 
\cite{kolokoltsov2013nonlinear} are among  the first to study  this type of equations.  \cite{crisan2014conditional}  motivated such a model \eqref{sde0} from a financial viewpoint, and proved the existence of a solution, but left uniqueness as an open problem.
\cite{kolokoltsov2013nonlinear} then established the uniqueness of \eqref{sde0} under differentiable and Lipschitzian conditions on $\sigma$ and $F$, and under the uniform ellipticity condition on  $a:=\sigma\sigma^*$, namely, there exists constant $\Lambda>0$, such that
\begin{align}\label{unelp}
\Lambda^{-1}| \xi|^2  \leq |\xi,  a(t,y,x) \xi | \leq \Lambda |\xi|^2,\quad \forall\,  (t,y,x)\in \RR_+\times \RR^d\times \RR^d, \xi\in \RR^d\,, 
\end{align}
the well-posedness of equation \eqref{sde0} was obtained.

Our main contribution of this paper is to prove the pathwise uniqueness for the quantile-dependent SDE under the  weak H\"{o}rmander condition (see e.g. \cite[Page 355]{hu2016analysis}), which is a much different condition than the uniform ellipticity condition \eqref{unelp}. Moreover, many SDEs including some used in financial markets fail to satisfy the uniform ellipticity condition so alternative conditions like ours, which require substantially different proofs, are important. To control the quantile when proving the uniqueness,  we require that the solution $X_t$ (as random vector)  has density (with respect to Lebesgue measure)  and this density is strictly positive with a certain decay property. This problem of existence of density is an important topic in probability theory and partial differential equations. The weak H\"{o}rmander condition imposed in the McKean-Vlasov equation ensures the existence of the density of its solution. Let us recall one such result on the following $nd$-dimensional Langevin-type stochastic differential equation: 
\begin{align}\label{chaineq}
\begin{cases}
&dX^{1}_t = F_1\(t, Q_{\alpha}(X_t), X^{1}_t, \cdots, X^{n}_t\)dt + \sigma\(t, Q_{\alpha}(X_t), X^{(1)}_t, \cdots, X_t^{(n)}\)dW_t,  \\
&dX^{2}_t = F_2\(t, Q_{\alpha}(X_t), X^{1}_t, \cdots, X_t^{n}\)dt, \\
&dX^{3}_t = F_3\(t, Q_{\alpha}(X_t), X^{2}_t, \cdots, X_t^{n}\)dt,   \\
&\qquad \vdots   \\
&dX^{n}_t = F_n\(t, Q_{\alpha}(X_t), X^{{n-1}}_t, X_t^{n}\)dt,
\end{cases}
\end{align}
where $d$ and $n$ are  positive integers;    $(W_t)_{t\ge 0}$ is a standard $d$-dimensional Brownian motion;   $X^{(i)},1\leq i \leq n,$ are all $d$-dimensional processes, 
and $(X_t)_{t\geq 0}=(X_t^{1},\dots,X_t^{n})_{t\geq 0}$;    $F_1: \RR_+\times \RR^{nd}\times \RR^{nd}\to \RR^d$;  $F_i:\RR_+\times \RR^{nd}\times \RR^{(n-i+2)d}\to \RR^d$ for $i=2, \cdots, n$;  and $\sigma:\RR_+\times \RR^{nd}\times \RR^{nd}\to \RR^d\otimes \RR^d$ are continuous functions.  
  Denote by $I_d$ and $0_d$ the $d\times d$ identity and zero matrices  respectively.  
Introducing $D=(I_d, 0_d, \cdots, 0_d)^T :  \in \RR^{nd\times d}$, and letting  $F=(F_1,\cdots,F_n)^T$, we can rewrite \eqref{chaineq} in the following abbreviated form
\begin{align}\label{simeq}
dX_t = F(t, Q_{\alpha}(X_t), X_t)dt + D\sigma(t, Q_{\alpha}(X_t), X_t)dW_t.
\end{align}
The system of equations \eqref{chaineq}  (or \eqref{simeq}) is highly degenerate if $n\geq 2$ and   
the elliptic condition \eqref{unelp}  
is obviously not satisfied. 
Still, in the special case that 
  $F$ and $\sigma$ in \eqref{simeq} are independent of the quantile, 
namely, when  \eqref{simeq}  is reduced  to 
\begin{align}\label{chain}
dX_t = \bar{F}(t, X_t)dt +D\bar{\sigma}(t, X_t)dW_t\,, 
\end{align} 
the  existence of the density, its derivatives and  its   two-sided Gaussian  bounds    have been  obtained  in
 \cite{de2018strong,  delarue2010density, menozzi2011parametrix, pigato2020density},  which are critical to this  work. 

The  degenerate stochastic differential equations  of the 
form \eqref{chain}  have been attracted more and more attention in the past  years 
(see e.g. \cite{priola2015weak, veretennikov2020weak, wang2016degenerate}).
When a Newton equation $\ddot x(t)=F(t, x(t), \dot x(t))$  is under influence of some uncertainty, the corresponding  stochastic differential equation could  be 
$\ddot x(t)=F(t, x(t), \dot x(t))+ G (t, x(t), \dot x(t))\dot W(t)$.  
This equation is of the form \eqref{chain} if we let $x_1(t)=x(t)$ and $x_2
(t)=\dot x(t)$, namely, 
$d x_1(t)=F(t, x_1(t),   x_2(t))+ G (t, x_1(t),   x_2(t))d W(t)$ and
$dx_2(t) =x_1(t) dt$. 
%
There are also many other examples.  Equation \eqref{chain} also 
corresponds to the dynamics of a finite-dimensional non-linear Hamiltonian system (a chain of anharmonic oscillators) coupled with two heat reservoirs at different temperatures, which was used by Eckmann et al. \cite{eckmann1999non} (see also \cite{herau2004isotropic, soize1994fokker, talay2002stochastic}) to study the statistical mechanics of such system. Rey-Bellet and Thomas \cite{rey2000asymptotic} considered the low temperature asymptotic behavior of the invariant measure in the framework of \eqref{chain}. Additionally, there are some applications of the Langevin-type equation in pricing Asian options (see e.g. \cite{barucci2001some}).

To obtain the existence and uniqueness of equation \eqref{simeq}, we use the fixed point theorem. 
But, to apply the fixed point theorem, we need to bound a certain  distance 
between $Q_\al(h_1)$ and $Q_\al(h_2)$  by 
 a certain distance of $h_1$ and $h_2$ {\blue (see \eqref{ql1})}. 
This was already done in \cite{kolokoltsov2013nonlinear}. 
We also need to bound the distance between 
 $u^{(1)}$ and $u^{(2)}$ by the distance between $\om^{(1)}$ and 
 $\om^{(2)}$, where each $u^{(i)}$  ($i=1, 2$) is the density of $X_t^{(i)}$
in \eqref{chaineq}
when $Q_\al( X_t^{(i)})$ 
 is replaced by $\om^{(i)}$ { (see Proposition \ref{sensiofu})}.  
This is relatively complicate and requires the fact that the density $u$ of the solution $X_t$  of 
 \eqref{simeq} is characterized by the corresponding    Fokker-Planck equation.   
Thus, the above problem  of controlling the distance between
 $u^{(1)}$ and $u^{(2)}$ by the distance between $\om^{(1)}$ and 
 $\om^{(2)}$ is reduced to how  the solution 
of the corresponding Fokker-Planck equation depends on the coefficients. 
However, because of the degeneracy,  it is hard to use the PDE approach 
as in \cite{kolokoltsov2013nonlinear}.   
Instead, we shall use the time-dependent Feynman-Kac formula.

The paper is organized as follows. In Section \ref{s.2}, we present the main hypotheses and main results of this paper as well as some notation. 
Some useful apriori estimates on the density including the tail estimates and lower bounds are given in Section \ref{s.3}.  
 These two-sided bounds   and the Feynman-Kac formula play a central role in the whole article. We present in Section \ref{s.4} the proof of our main results. 

\section{Main results}\label{s.2}
For any $x\in \RR^{nd}, $ we write  $x = (x_1, \cdots, x_n)=(x_1^1,\cdots,x_1^d; \cdots; x_n^1,\cdots,x_n^d)$, where for $i=1,\cdots,n$, $j=1,\cdots,d$, $x_i\in\RR^d$, $x_i^j\in\RR$. 
Let $|x_i|$ denote the Euclidean norm of $x_i$, that is, $|x_i|=
\left( \sum_{ j=1}^{d}|x_i^j|^2 \right)^{\frac{1}{2}}$. Let $F_i: \RR_+\times \RR^{nd} \times 
\RR^{((n-i+2)\wedge n)d} \rightarrow \RR^d$ be continuous mappings.  
For notational simplicity, we may consider $F_i$ as a 
continuous mapping from $\RR_+\times \RR^{nd} \times 
\RR^{nd} $ to $ \RR^d$ and write $F=(F_1\cdots,F_n)=(F_1^1,\cdots,F_1^d; \cdots; F_n^1,\cdots,F_n^d)$ as well.

For any $d \times d$ matrix $a=(a_{ij})_{i,j=1}^d$, denote  by  $|a|= (\sum_{i,j=1}^{d}|a_{ij}|^2  )^{\frac{1}{2}}$  its Hilbert-Schmidt norm. 
In what follows, we use $\|\cdot \|_p$ for the $L^{p}$ norm on the probability space
$(\Omega, \cF, \PP)$. For any measurable function $f$ on a  Euclidean space, $|f|_{L^p}$ denotes the $L^p$ norm of $|f|$ with respect to the Lebesgue measure.

The notation $\nabla$ stands for the gradient with respect to all 
space variables. Let $f \in C(\RR^+ \times \RR^{nd} \times \RR^{nk}, \RR^d)$, $k =1, \cdots, n$. $\nabla_{x_i} f(t, y, x)$ denotes the gradient 
operator   w.r.t. the $i$th  
space variable $x_i \in \RR^d$, which is a $d \times d$ Jacobian matrix.
%


Fix a time horizon $[0,T]$. We will need  the following  
hypotheses  
for coefficients $F$ and $\sigma$ and the initial condition $X_0$. 
\begin{enumerate}
	\item[({\bf H1})] $F$ is uniformly bounded at the origin of the third argument. That is, there exists a positive constant $\kappa$ such that
	\[
	  \sup_{t\in [0,T],y\in \RR^{nd}}|F (t,y, 0)|\leq \kappa <\infty.
	\] 
    \item [({\bf H2})] The function $a:=\sigma\sigma^*$ is uniformly elliptic,  namely,    \eqref{unelp} is satisfied. 

	\item[({\bf H3})] $F$ and $\sigma$ are uniformly Lipschitz continuous in space variables with constant $\kappa>0$, i.e., for all $y, \bar{y}, x, \bar{x} \in \RR^{nd}$,
	\begin{align}
	\sup_{t\in [0,T]}\big(|F(t,y, x) - F(t, \bar{y}, \bar{x})|+ |\sigma(t,y, x) - \sigma(t, \bar{y}, \bar{x})|\big)\le \kappa \(|x - \bar{x}| + |y - \bar{y}|\).
\end{align}


	\item[({\bf H4})] 
	The function $x \mapsto  F(t,y,x)$ is twice differentiable and function $x_1  \mapsto a(t,y,x_1,\dots, x_n) $ is three times differentiable. Moreover, 
 the following inequalities hold
	\begin{align*}
\sup_{(t,y,x)\in [0,T]\times \RR^{nd}\times \RR^{nd}}\bigg|\sum_{j,k=1}^d\frac{\partial^2}{\partial_{x^j_{1}}\partial_{x^k_{1}}}a_{jk}(t,y,x)\bigg|\leq \kappa,
	\end{align*}
	\begin{align*}
	\sup_{t\in [0,T]} \sum_{j,k=1}^d\bigg|\frac{\partial^2}{\partial_{x^k_{1}}\partial_{x^j_{1}}}a_{kj}(t,y,x) - \frac{\partial^2}{\partial_{x^k_{1}}\partial_{x^j_{1}}}a_{kj}(t,\bar{y},\bar{x})\bigg| \le \kappa\(|x - \bar{x}| + |y - \bar{y}|\),
	\end{align*}
	and
	\begin{align*}
	\sup_{t\in [0,T]} \sum_{i=1}^n\sum_{j=1}^d\bigg| \frac{\partial}{\partial x^j_{i}}F^j_{i}(t,y, x) - \frac{\partial}{\partial x^j_{i}}F^j_{i}(t,\bar{y},\bar{x}) \bigg| \le 
	\kappa\(|x - \bar{x}|+|y - \bar{y}| \),
	\end{align*}
	for all $x,\bar{x},y,\bar{y}\in \RR^{nd}$.

	\item[({\bf H5})]  For any integer $i= 2,\dots, n$, the derivative $\nabla_{x_{i-1}}F_i(t,y,x)$ is $\eta$-H\"older continuous in the first spatial variable $x_{i-1}$ with constant $\kappa$, and there exists a closed convex subset $\varepsilon_{i-1}$ contained in the set of invertible $d \times d$ matrices, such that for all $t \in [0,T]$ and $(x_{i-1}, \cdots,x_n) \in \RR^{(n-i+2)d}$, the matrix $\nabla_{x_{i-1}}F_i(t, y,x_{i-1}, \cdots, x_n)$ belongs to $\varepsilon_{i-1}$. 
\item[({\bf I})]\label{init} $X_0$ is a random variable independent of $W$. The probability law of $X_0$ has a continuously differentiable density  $f>0$ satisfying the following integrability condition
\begin{align}\label{hypi}
	U=&\int_0^{\infty} \sup_{|z|\geq r}|f(z)|^2(r^{4n-1+\ep}+r^{n-1})dr \nonumber\\
	&\quad  +\int_0^{\infty} \Big[\sup_{|z|\geq \lambda}|\nabla f(z)|^4\Big]\(\lambda^{4nd-1+\ep} +\lambda^{nd-1}\) d\lambda<\infty,
	\end{align}
for some constant $\ep>0$.  
\end{enumerate}

\begin{remark}
	The most important hypotheses in this work 
	 are the hypotheses (H2) and (H5): the matrices $(\nabla_{x_{i-1}}F_i)_{2 \le i \le n}$ have full rank, which imply  a version of the (weak) H\"{o}rmander condition. It ensures the existence of the probability density of the solution to \eqref{simeq}. Let us point out that in (H2) we assume that $a=\sigma \sigma^*$ is uniformly elliptic.   However,  the 
diffusion coefficient	 $D\sigma\sigma^* D^*$  of   the whole system \eqref{simeq} 
is highly degenerate. 
\end{remark}

\begin{remark}
Hypotheses (H3), (H4) 
are  to guarantee the Lipschitz continuity of the  function $c$ defined in \eqref{termc} below. In addition, they imply that $c$ is bounded, which is needed  in the application of the Feynman-Kac formula (see Theorem \ref{thmfk}).
\end{remark}
%
 
\begin{remark}
	At the first look, hypothesis  (I) seems a little complicated. However, Gaussian densities and many other functions satisfy this condition. 
Furthermore, it is worth mentioning that to prove Proposition \ref{sensiofu} (i.e. 
	the local existence and uniqueness), 
	hypothesis (I) can be weakened  to  the following form: 
	\begin{align*}
	\int_{\RR^{nd} } f(y)^2 \(|y|^{{nd}+\ep} +1\) dy +\int_0^{\infty} \Big[\sup_{|z|\geq \lambda}|\nabla f(z)|^4\Big]\(\lambda^{4nd-1+\ep} +\lambda^{nd-1}\) d\lambda<\infty.
	\end{align*}
The   condition \eqref{hypi} is  used  to guarantee   the  global  existence 
and uniqueness of    \eqref{simeq}. 
\end{remark}

In the next theorem, we provide the existence and uniqueness result for equation \eqref{simeq}, which  is the main result of this paper.
\begin{theorem}\label{SDEsu}
	Assume   hypotheses  (H1)-(H5)  
and hypothesis (I).
Then,  there exists a unique  strong solution to SDE \eqref{simeq} on $[0,T]$.
\end{theorem}

The idea to prove Theorem \ref{SDEsu} follows the spirit of 
\cite{kolokoltsov2013nonlinear} to construct a contraction mapping associated to equation \eqref{simeq}. To this end, we need to introduce an auxiliary equation. Given a continuous   function $\omega$ on $[0,T]$ with values in $\RR^{nd}$, we consider the following stochastic differential equation
\begin{align}\label{auxiliaryeq}
dX_t^{\omega} = F(t, \omega_t, X_t^{\omega})dt + D\sigma(t, \omega_t, X_t^{\omega})dW_t,
\end{align}
with initial condition $X_0$ satisfying hypothesis (I). Under hypotheses (H1)-(H5), equation \eqref{auxiliaryeq} has a unique solution, whose density exists and  satisfies the following Fokker-Planck equation
\begin{align}\label{omega_pde}
\frac{\partial}{\partial t} u^{\omega}_t(x)=\frac{1}{2}\sum_{i,j=1}^d a_{ij}(t,\omega_t,x)\frac{\partial^2}{\partial x_1^i\partial x_1^j}u^{\omega}_t(x)+\langle b(t,\omega_t,x), \nabla u^{\omega}_t(x)\rangle+c(t,\omega_t,x)u^{\omega}_t(x),
\end{align}
with initial condition $u_0^{\omega}=f$, where 
\begin{align}\label{terma}
a = (a_{ij})_{i,j=1}^d= \sigma\sigma^*,
\end{align}
 $b=(b_1,\dots, b_n)$ with $b_i=(b_i^1,\dots,b_i^d)$, $i=1,\dots, n$ and
\begin{align}
b_i^j(t,y,x)=-F_i^j(t,y,x)+\mathbf{1}_{\{i=1\}}\sum_{k=1}^d\frac{\partial}{\partial x_1^k}a_{kj}(t,y,x),
\end{align}
and
\begin{align}\label{termc}
c(t,y,x)=-\sum_{i=1}^n\sum_{j=1}^d\frac{\partial}{\partial x_i^j}F_i^j(t,y,x)+\frac{1}{2}\sum_{j,k=1}^d\frac{\partial^2}{\partial x_1^j\partial x_1^k} a_{jk}(t,y,x),
\end{align}
for any $(t,y,x)\in [0,T]\times \RR^{nd}\times \RR^{nd}$. 
Similarly, if $\eqref{simeq}$ has a solution $X_t$ with quantile $Q_{\alpha}(X_t)$ being continuous in time, then the law of $X_t$ has a density $u$ that is the solution to the following equation 
\begin{align}\label{quan_pde}
\frac{\partial}{\partial t} u_t(x)=&\frac{1}{2}\sum_{i,j=1}^d a_{ij}(t,Q_{\alpha}(u_t),x)\frac{\partial^2}{\partial x_1^i\partial x_1^j}u_t(x)+\langle b(t,Q_{\alpha}(u_t),x), \nabla u_t(x)\rangle\nonumber\\
&+c(t,Q_{\alpha}(u_t),x)u_t(x).
\end{align}

Thus, the proof of Theorem \ref{SDEsu} is reduced  to prove  that   \eqref{quan_pde} admits a unique solution. However, it is not easy to deal with such PDE whose coefficients depend on  quantiles.
We shall find an appropriate  Banach space $\BB$ and 
construct a mapping $\cM$ on $\BB$.
The well-posedness of \eqref{quan_pde} is shown by proving that $\cM$ is a 
contraction map on $\BB$ in Proposition  \ref{lcPDEthm} (below). 

In the next theorem, we prove the well-posedness of   \eqref{quan_pde}.
\begin{theorem}\label{PDEthm}
	Let $f$ be a continuous differentiable function on $\RR^{nd}$ satisfying hypothesis (I). Assume   hypotheses (H1)-(H5).
	Then, there exists a function $u$ on $[0,T]\times \RR^{nd}$, which is the unique solution to PDE \eqref{quan_pde} with initial condition $f$.
\end{theorem}


In \cite{kolokoltsov2013nonlinear}, to obtain the stability 
result  the author uses the two-sided bounds   of    the density  and its    first order derivatives  under 
uniform  ellipticity condition. In our hypoellipticity  case,
we   encounter the following difficulties. 
\begin{enumerate}
	\item[(1)]  For   \eqref{chain},  Pigato \cite{pigato2020density}  obtained upper bounds for  the derivatives of transition density of any order. The first derivative with respect to the variable $x_i^j$ $(i=1,\cdots,n; j=1,\cdots,d)$ is given by 
	\begin{align*}
	|\partial_{x_i^j} p(t,x;0,y)| \le \frac{C}{t^{(2\lfloor\frac{ij-1}{d}\rfloor+1+n^2d)/2}}\exp\(-\frac{|\mathcal{T}_t^{-1}(x-\theta_t(y))|^2}{C}\),
	\end{align*} 
	where $y$ is the initial position, $C$ is a constant,  $\lfloor\cdot\rfloor$ denote the integer part function,  $\mathcal{T}$ and $\theta_t$ are  given by \eqref{tt} and \eqref{deterministic} below.  As we see, this bound is more singular  near $t=0$ than  that in the  elliptic case. 

\item[(2)] 
To overcome this problem, we  assume 
that the initial condition $f$ satisfies  certain differentiability and integrability
(over the whole $\RR^{nd}$) conditions,  
%
in hope that the singularity difficulty can be absorbed in the initial condition. However, proceeding with this effort, we immediately encounter the difficulty  that we don't know how to pass the gradient   $  \nabla_x p(t,x;s,y)$ to $\nabla  f$ in the following integral: 
\[
\int_{\RR^{nd}}f(y)\nabla_x p(t,x;s,y)dy 
\]
since $p(t,x;s,y)$ is not the form of $p(t,s, x-y)$.  To get around this difficulty, we apply the time inhomogeneous Feynman-Kac formula.
This enables us    to finish the stability analysis of the solution to \eqref{omega_pde} with respect to $\omega$. 
\end{enumerate}


%

\section{A priori  estimates of the density}\label{s.3}
 
In the remaining part of the    paper, we assume that $d=1$ to simplify the presentation.  The case $d>1$ can be treated analogously  with only   additional notational complexity.  We use $C>0$ to denote  a generic constant  which may vary from occurrence  to occurrence. 
 
First, let us turn to  \eqref{chain}.  We need a result from \cite{de2018strong}. 
To state this result, we need to introduce  the following conditions on  the coefficients $\bar{F}$ and $\bar{\sigma}$. 
\begin{description}
	\item[(C1)] $\bar{F}(t,0)$ is bounded for all $t\in [0,T]$  and  $\bar{a}=\bar{\sigma}\bar{\sigma}^*$ is uniformly  elliptic with the positive constant  $\Lambda$. 
	
	\item[(C2)] $\bar{\sigma}$ is globally  Lipschitz in the space variable
	uniformly in   time
	variable. For all $j=1, 2, \cdots, n $
	the functions $\bar{F}_i, i=1,   \cdots,  j $,  are uniformly $\eta_j$-H\"older continuous in the $j$th spatial variable with $\eta_j \in (\frac{2j-2}{2j-1},1]$, uniformly in time and other spatial variables.
	
	\item[(C2')] The functions $\bar{F}_1, \cdots, \bar{F}_n$ and $\bar{\sigma}$ are   uniformly Lipschitz and $\eta$-H\"older continuous ($\eta \in (0, 1]$) with respect to the underlying space variables  respectively. 

	\item[(C3)] For each integer $2 \le i \le n$, $(t,  x_i, \cdots, x_n ) \in \RR_+ \times \RR^{(n-i+1)d}$, the function $x_{i-1} \in \RR^d \mapsto \bar{F}_i(t, x_{i-1}, \cdots, x_n)$ is    continuously differentiable  and its  derivative, denoted by $(t, x_{i-1}, \cdots, x_n) \in \RR_+ \times \RR^{(n-i+2)d} \mapsto \nabla_{x_{i-1}}\bar{F}_i(t, x_{i-1}, \cdots, x_n)$, is  $\eta$-H\"older continuous in the first space variable $x_{i-1}$  with constant $\kappa$. Moreover, there exists a closed convex subset $\varepsilon_{i-1}$ contained in the set of invertible $d \times d$ matrices, such that for all $t\ge0$, $i=2,\cdots,n$ and $(x_{i-1}, \cdots, x_n) \in \RR^{(n-i+2)d}$, the matrix $\nabla_{x_{i-1}}\bar{F}_i(t, x_{i-1}, \cdots, x_n)$ belongs to $\varepsilon_{i-1}$. 
\end{description}
\begin{remark}\label{imply}
	Conditions (C1), (C2), (C2') and (C3) can be easily verified by  hypotheses (H1)-(H5). In fact,  (C1) and (C3) are the same as   (H1), (H2) and (H5). Additionally, hypotheses  (H3) and (H4) imply   (C2) and (C2').
\end{remark}
\begin{theorem}[see \cite{de2018strong}]\label{strongsol}
	Assume  (C1), (C2) and (C3).  There exists a unique strong
	solution to   SDE \eqref{chain}.
\end{theorem}

We also need a result about the    Gaussian estimate for the density of the solution to   \eqref{chain}.
To state this result  we introduce the scale matrix $\mathcal{T}$ and shift vector $\theta$ as follows. Fix $t\geq 0$ and  $x\in\RR^{nd}$.   Let  $\mathcal{T}_t$ denote 
the following  $nd \times nd$ diagonal matrix: 
\begin{align}\label{tt}
\mathcal{T}_t=\begin{pmatrix}
\mathcal{T}_t^1 & 0 &\dots & 0\\
0 & \mathcal{T}_t^2  &\dots & 0\\
\vdots & \vdots & \ddots &\vdots\\
0 & 0 & \dots &\mathcal{T}_t^n
\end{pmatrix}=\begin{pmatrix}
t^{\frac{1}{2}}I_d & 0 &\dots & 0\\
0 & t^{\frac{3}{2}}I_d  &\dots & 0\\
\vdots & \vdots & \ddots &\vdots\\
0 & 0 & \dots &t^{n-\frac{1}{2}}I_d
\end{pmatrix}.
\end{align}
Let $\theta_\cdot (x): [0,T]\mapsto\RR^{nd}$ be the solution to following (deterministic) ODE,
\begin{align}\label{deterministic}
\begin{cases}
\frac{d}{dt}\theta_t(x) = \bar{F}\(t,\theta_t(x)\),\\
\theta_0(x) =  x.
\end{cases}
\end{align}
\begin{theorem}[see \cite{delarue2010density}]\label{tse}
    Assume (C1), (C2') and (C3).
	Let $X$ be the solution to \eqref{chain} with initial condition $X_0=x$, where $x\in \RR^{nd}$. Then,  for any $t\in [0,T]$, the law of  $X_t$  admits a probability density, denoted by $p_t(\cdot,x)$. Moreover, there exists a constant $C_T \ge 1$, depending on $T, n, d, \Lambda, \eta,  $
	the Lipschitz constants in (C1)-(C3),   and $\varepsilon_1, \varepsilon_2, \cdots, \varepsilon_{n-1}$, 
	such that for any $y\in \RR^{nd}$,
	\begin{align}\label{lower0}
	\frac{1}{C_Tt^{n^2d/2}}\exp \( -C_T |\mathcal{T}_t^{-1}(\theta_t (x) - y)|^2\) \le p_t(y,x) \le 	\frac{C_T}{t^{n^2d/2}}\exp \( -C_T^{-1} |\mathcal{T}_t^{-1}(\theta_t (x) - y)|^2\),
	\end{align}
	where $\mathcal{T}_t$ and $\theta_t(x)$ are given by  \eqref{tt} and \eqref{deterministic} respectively.
\end{theorem}

\begin{remark}  We still cite the theorem for general dimension $d$. 
	However, we will continue to work on the case $d=1$. 
\end{remark}

Let $F: \RR^n\rightarrow \RR^n$ be defined as in Section  \ref{s.1}
 (with $d=1$). 
For any continuous function $\omega:[0,T]\to \RR^n$, and $x\in \RR^n$, we define, analogously to \eqref{deterministic}, a function $\theta^{\omega}=(\theta^{\omega}_1,\dots,\theta^{\omega}_n)$ on $[0,T]$ with values in $\RR^n$ by the following ODE
\begin{align}\label{determinF}
\begin{cases}
\frac{d}{dt}\theta^{\omega} (t,x) = F\(t,\omega_t, \theta^{\omega} (t,x)\), t\in [0, T]\,,
\\
\theta^{\omega} (0,x)=x\,.
\end{cases}
\end{align}
We have  the following lemma about $\theta^{\omega}$.
\begin{lemma}\label{lmma}
Assume hypotheses (H1)-(H5)   and  assume that   $\omega$ is a  continuous function  of $t\in [0, T]$.  
Let $\theta^{\omega}$  satisfy  \eqref{determinF}. Then, for $0 \leq t\leq T$ and $x\in \RR^n$, the following inequalities hold
\begin{align}\label{soltheta}
e^{-\kappa t}|x|-\kappa t\leq |\theta^{\omega}(t,x)|\leq  \(|x|+\kappa t\)e^{\kappa t}
\end{align}
and
\begin{align}\label{lm2}
e^{-n\kappa t}\leq \det (\nabla \theta^{\omega}(t, x))\leq e^{n\kappa  t},
\end{align}
where $\kappa$ is the  positive constant that appeared  in hypotheses (H1)-(H5).

\end{lemma}
\begin{proof}
By the Lipschitz property and uniformly boundedness (at the origin) of $F$, we see  that
\begin{align*}
|\theta^{\omega}(t, x)|=&\Big|x +\int_0^tF(r,\omega_r,\theta^{\omega}(r,x))dr\Big|\\
\leq &|x|+\int_0^t \( |F(r,\omega_r,0)| + \kappa|\theta^{\omega}(r,x)|\)dr\\
\leq &|x|+\kappa t+\int_0^t\kappa|\theta^{\omega}(r,x)|dr.
\end{align*}
An application of Gronwall's inequality  yields 
\begin{align}
|\theta^{\omega}(t,x)|\leq \(|x|+\kappa t\)e^{\kappa t}.\label{e.3.4}
\end{align}
This proves the second inequality in \eqref{soltheta}. To prove the first inequality, we consider the following backward ODE:
\begin{align}\label{bode}
\begin{cases}
\frac{d}{ds}\hat{\theta}_s =  -  F(s, \omega_s, \hat{\theta}), \ \ 0\le s<t, \\
\hat{\theta}_t=\xi\in \RR^n.
\end{cases}
\end{align}
Similar to \eqref{e.3.4}, we can show that
\[|\hat{\theta}_s|\leq  \(|\xi|+\kappa(t-s)\)e^{\kappa (t-s)},\]
for all $s\in [0,t]$. Notice that $\hat{\theta}=\{\theta^{\omega}(s,x)\,; \ s\in (0,t)\}$ is the solution to \eqref{bode} with terminal condition $\hat{\theta}_t=\theta^{\omega}(t,x)$. 
Then,  we have 
\[
|x|=\hat{\theta}_0\leq \(|\theta^{\omega}(t,x)|+\kappa t\)e^{\kappa t}.
\]
The proof of inequality \eqref{soltheta}  is then completed.

Taking the derivative of the following equation with respect to $x$,
\[
\theta^{\omega}(t, x)=x+\int_s^t F(r,\omega_r, \theta_r(x))dr,
\]
we have
\[
\nabla \theta^{\omega}(t, x)=I_d+\int_s^t\nabla F(r, \omega_r, \theta^{\omega}(r, x))\nabla \theta^{\omega}(r, x)dr.
\]
By Liouville's formula, we can write
\begin{align}\label{lm21}
\det(\nabla \theta^{\omega}(t, x)) = \exp\(\int_s^t {\rm tr}[\nabla F\(r, \omega_r, \theta^{\omega}(r, x)\)] dr\).
\end{align}
Now the hypothesis  (H3) can be applied to  obtain  \eqref{lm2}. The lemma is then proved.  
\end{proof}	
By Lemma \ref{lmma} and the implicit function theorem, we have the following corollary.
\begin{corollary}\label{coro}
Assume hypotheses (H1)-(H5) and assume   $\omega(t)$, $0\leq t\leq T$  is a continuous function. Let $\theta^{\omega}$  satisfy  \eqref{determinF}.  Then,  there exist a function $(\theta^{\omega})^{-1}(t,\cdot)$ such that
\[
\theta^{\omega}\big(t,(\theta^{\omega})^{-1}(t,x)\big)=(\theta^{\omega})^{-1}\big(t,\theta^{\omega}(t,x)\big)=x,
\]
for all $x\in \RR^n$. Moreover, the gradient of $(\theta^{\omega})^{-1}$ 
with respect to the space variable  satisfies the following inequality: 
\begin{align}\label{lm3}
e^{-n\kappa t}\leq \det (\nabla (\theta^{\omega})^{-1}(t, x))\leq e^{n\kappa t}.
\end{align}
\end{corollary}




In  the next proposition, we provide a tail estimate for the solution to \eqref{omega_pde}.
\begin{proposition}\label{epsilon}
Assume hypotheses (H1)-(H5). Let $f$ be a continuous integrable function on $\RR^{n}$. Then for any $\varepsilon>0$, there exist $K>0$ such that 
\begin{align}
\int_{\mathcal{G}_K} |u^{\omega}_t(x)|dx \le \varepsilon,
\end{align}
 for any $t\in [0,T]$ and for any  continuous function $\omega:[0,T]\to\RR^{n}$, where 
 \[
 \mathcal{G}_K=\Big\{x=(x_1,\dots,x_n) \in \RR^n: \max_{1\leq i\leq n}|x_i | \ge K \Big\}.
 \]
\end{proposition}
\begin{proof}
For any $\varepsilon>0$, due to integrability of $f$, we can choose $\bar{K}$ such that 
\begin{align}\label{bkf}
\int_{|x|\ge \bar{K}} |f(x)|dx \le \varepsilon.
\end{align}
Denote by $p^{\omega}_t(x,y)$ the  transition 
 density of $X^{\omega}$ from $y$ at time $0$ to $x$ at time $t$. Then, it is well-known 
\[
u^{\omega}_t(x) = \int_{\RR^n} p^{\omega}_t(x,\xi)f(\xi) d\xi.
\]
Notice that hypotheses (H1)-(H5) ensures   that functions $\bar{\sigma}$ and $\bar{F}$ given by
\[
\bar{\sigma}(t,x)=\sigma(t,\omega_t,x),\ \mathrm{and}\ \bar{F}(t,x)=F(t,\omega_t,x),
\]
for all $(t,x)\in [0,T]\times \RR^n$, satisfy conditions (C1)-(C3) and (C2'). Additionally,  the independence   of $t$ and $x$  of the constant of $\kappa$ 
in  hypotheses (H1)-(H5) and Remark \ref{imply} imply  that  the constant $C_T$ appearing  in Theorem \ref{tse} is independent of the choice of $\omega$. This allows us to apply Theorem \ref{tse}  to obtain 
\begin{align*}
\int_{\mathcal{G}_K}|u^{\omega}_t(x)|dx&= \int_{\mathcal{G}_K}\left|\int_{\RR^n} p^{\omega}_t(x,\xi)f(\xi)d\xi \right|dx \\
&\le \int_{\mathcal{G}_K}\int_{\RR^n} C_Tt^{-n^2/2} \exp \( -C_T^{-1} |\mathcal{T}_t^{-1}(x - \theta^{\omega} (t,\xi))|^2\)|f(\xi)|d\xi dx\\
&= \int_{\mathcal{G}^1_K}\int_{\RR^n} C_Tt^{-n^2/2} \exp \( -C_T^{-1} |\mathcal{T}_t^{-1}(y)|^2\)|f(\xi)|d\xi dy,
\end{align*}
where $\theta^{\omega}$ is defined by  \eqref{determinF} and
\[
\mathcal{G}^1_K=\{y\in \RR^{n}:\max_{1\le i \le n} |y_i+\theta^{\omega}_i (t,\xi)|>K\}.
\]
Performing a change of variable $z=\mathcal{T}_t^{-1}(y)$, we have
\begin{align*}
\int_{\mathcal{G}_K}&|u^{\omega}_t(x)|dx \le \int_{\mathcal{G}^2_K}\int_{\RR^n}C_T\exp \( -\frac{|z|^2}{C_T}\)f(\xi) d\xi dz \\
&\le \int_{\RR^n}\int_{\{|\xi| >\bar{K}\}}C_T\exp \( -\frac{|z|^2}{C_T}\)f(\xi) d\xi dz + \int_{\mathcal{G}^2_K } \int_{\{ |\xi|\le \bar{K}\}} C_T \exp \( -\frac{|z|^2}{C_T}\)f(\xi)  d\xi dz,
\end{align*}
where
\[
\mathcal{G}^2_K=\Big\{z\in \RR^{n}:\max_{1\leq i\leq n}|t^{i-\frac{1}{2}}z_i+\theta^{\omega}_i (t,\xi))|>K\Big\}.
\]
As a consequence of \eqref{bkf}, we have 
\begin{align*}
\int_{\RR^n}\int_{\{|\xi| >\bar{K}\}}C_T\exp \( -\frac{|z|^2}{C_T}\)f(\xi) d\xi dz\leq (\pi C_T)^{\frac{n}{2}}C_T\epsilon.
\end{align*}
On the other hand, by Lemma \ref{lmma}, we know that 
\begin{align*}
|\theta^{\omega}(t,\xi)|\leq  \(|\xi|+\kappa t\)e^{\kappa t}\leq \bar{K}e^{\kappa T}+\kappa Te^{\kappa T},
\end{align*}
for all $|\xi|\leq \bar{K}$ and $t\in [0,T]$. Then, we have 
\[
\mathcal{G}^2_K\subseteq\Big\{z\in \RR^{n}:\max_{1\leq i\leq n}|t^{i-\frac{1}{2}}z_i|>K-\bar{K}e^{\kappa T}-\kappa Te^{\kappa T}\Big\} =: \mathcal{G}^3_K.
\]
Therefore, for any $\epsilon>0$,  there exists $K$ sufficiently  large such that
\begin{align*}
\int_{\mathcal{G}^2_K } \int_{\{ |\xi|\le \bar{K}\}}C_T \exp \( -\frac{|z|^2}{C_T}\)f(\xi)  d\xi dz\leq \int_{\{ |\xi|\le \bar{K}\}}C_Tf(\xi)  d\xi\int_{\mathcal{G}^3_K }  \exp \( -\frac{|z|^2}{C_T}\) dz\le \epsilon.
\end{align*}
The proof of this proposition is complete.
\end{proof}

\begin{proposition}\label{plower}
Assume hypotheses (H1)-(H5). Let $f$ be a positive, continuously  integrable function on $\RR^{n}$. Then for any $K>0$ there exists $\delta >0$ such that 
\begin{align}
\inf \Big\{u^{\omega}_t(x):\max_{1\le j\le n}|x_j| \le K\Big\} \ge \delta,
\end{align}
for all $t\in [0,T]$ and  for all continuous functions $\omega$ on $[0,T]$
 with values  in $\RR^{n}$. 
\end{proposition}
\begin{proof}
Fix $K>0$. For any $x\in \RR$ with $|x|\leq K$. By  the lower bound of \eqref{lower0}  we get
\begin{align*}\label{domain}
u^{\omega}_t(x) &= \int_{\RR^n} p^{\omega}_t(x, \xi) f(\xi) d\xi \\
&\ge \int_{\mathcal{R} } p^{\omega}_t(x, \xi) f(\xi) d\xi \\
&\ge \frac{1}{C_Tt^{n^2/2}} \int_{\mathcal{R} } \exp \( -C_T |\mathcal{T}_t^{-1}(x - \theta^{\omega} (t, \xi))|^2\) f(\xi) d\xi,
\end{align*}
where \[\mathcal{R} =  \{ \xi\in\RR^n:  |x_1 - \theta^{\omega}_1(t, \xi)| \le \sqrt{t},\cdots,|x_n - \theta^{\omega}_n(t, \xi)| \le  t^{(2n-1)/2} , \max_{1\le j\le n}|x_j| \le K\}.\]
Due to Lemma \ref{lmma}, we know that
\[
e^{-\kappa t}|\xi|-\kappa t\leq |\theta^{\omega}(t,\xi)|,
\]
for all $t \in [0,T].$  
 Thus,     we have
\begin{align*}
|\xi_i|\leq |\xi|\leq  \(|\theta^{\omega}(t,\xi)|+\kappa t\)e^{\kappa t}\,. 
\end{align*}
For any $\xi=(\xi_1,\dots,\xi_n)\in \mathcal{R}$, it is easy to see 
\[
|\theta^{\omega}(t,\xi)|\leq \bigg[\sum_{i=1}^n\(t^{\frac{2i-1}{2}}+|x_i|\)^2\bigg]^{\frac{1}{2}}\leq \sqrt{2}|x|+\sqrt{2n}(T^{\frac{n-1}{2}}+1).
\]
This means that 
\begin{align}\label{rr1}
\mathcal{R} \subseteq \mathcal{R}^1 :=\{\xi \in \RR^n : |\xi|\leq (\sqrt{2}nK+\sqrt{2}n(T^{\frac{n-1}{2}}+1)+\kappa T)e^{\kappa T} \}.
\end{align}
Recall   that $f$ is a continuous positive integrable function. Hence, there exists $\delta>0$ such that $f(\xi)\geq \delta$ on the set $\mathcal{R}^1\supseteq \mathcal{R}$. As a consequence,  for any $x\in \RR^n$ with $|x|\leq K$, we have 
\begin{align*}
u^{\omega}_t(x) &\ge \frac{\delta}{C_Tt^{n^2/2}}\int_{\mathcal{R} }  \exp \( -C_T |\mathcal{T}_t^{-1}(x - \theta^{\omega}(t, \xi))|^2\)  d\xi.
\end{align*}
By change of variable $\mathcal{T}_t^{-1}(x - \theta^{\omega}(t, \xi))=y$ and then 
by Corollary \ref{coro}, we have 
\begin{align*}
u^{\omega}_t(x)&\ge \frac{\delta}{C_T} \int_{\{ y\in\RR^n:  |y_i| \le 1, i=1,\dots n \}}  \exp\( -C_T  |y|^2 \) \det \(\nabla (\theta^{\omega})^{-1}\(t, x-\mathcal{T}_t(y)\)\)dy \\
&\ge \frac{\delta}{C_T}  e^{-(nC_T+n\kappa T)}\int_{\{ y\in\RR^n:  |y_i| \le 1, i=1,\dots n \}}dy \nonumber \\
&= \frac{\delta}{C_T} 2^n e^{-(nC_T+n\kappa T)}, \nonumber 
\end{align*}
which completes the proof of the proposition.
\end{proof}

Combining the Propositions \ref{epsilon} and \ref{plower}, we arrive at the following 
result. 
\begin{proposition}\label{3conditions}
 Assume hypotheses (H1)-(H5) and that    $\omega:[0,T]\to \RR^n$
 is a   continuous function. Let $u^{\omega}$ be the solution to \eqref{omega_pde} with initial condition $f \in C(\RR^n) \cup L^1(\RR^n)$.    For any $\alpha \in (0,1)^n$ and $t\in [0,T]$, let $\hat{\omega}^{\alpha}=\hat{\omega}=(\hat{\omega}_1, \cdots, \hat{\omega}_n)=Q_{\alpha}(u^{\omega}_t)$ be the $\alpha$-quantile of $u^{\omega}_t$. Then, there exist  $K, \delta, \ep >0$, independent of $t$ and $\omega$, such that
\begin{align}\label{con1}
\int_{\big\{\displaystyle x \in \RR^n: \max_{1\le j\le n}|x_j| \ge K\big\}} |u^{\omega}_t(x)|dx \le \varepsilon,
\end{align}
\begin{align}\label{con2}
\max_{1\le j\le n}|\hat{\omega}_j| \le K,
\end{align}
and
\begin{align}\label{con3}
\inf\Big\{ u^{\omega}_t(x): \max_{1\le j\le n}|x_j| \le K\Big\} \ge \delta.
\end{align}
Fix $\alpha\in (0,1)^n$ and $K,\delta,\varepsilon>0$. Denote by $\mathcal{S}=\mathcal{S}_{\alpha, K,\delta,\varepsilon}$ the collection of density functions $h$ on $\RR^{n}$ such that
\begin{align}\label{sets}
\int_{\big\{\displaystyle x \in \RR^n: \max_{1\le j\le n}|x_j| \ge K\big\}} |h|dx \le \varepsilon,\ Q_{\alpha}(h)\leq K,\ \mathrm{and}\ \inf\Big\{ h(x): \max_{1\le j\le n}|x_j| \le K\Big\} \ge \delta.
\end{align}
Then, $\mathcal{S}$ is a   convex set.
\end{proposition}
\begin{proof}
Choose 
\[
0<\varepsilon < \min\( \alpha_1, \cdots, \alpha_n, 1-\alpha_1, \cdots, 1-\alpha_n\).
\]
Then, by Proposition \ref{epsilon}, there exists $K>0$ such that   \eqref{con1} is true. Inequality \eqref{con2} also holds true, due to the fact that
\[
\int_{-\infty}^{-K} dx_j  \int_{\RR^{n-1}}u^{\omega}_t(x) \prod_{i\ne j, 1\le i \le n}dx_i \leq \int_{\{x \in \RR^n: \max_{1\le j\le n}|x_j| \ge K\}} |u^{\omega}_t(x)|dx \le \varepsilon\leq \alpha_j,
\]
and
\[
\int_{K}^{\infty} dx_j  \int_{\RR^{n-1}}u^{\omega}_t(x)  \prod_{i\ne j, 1\le i \le n}dx_i\leq 1-\alpha_j,
\]
for all $j=1,\dots, n$.    \eqref{con3} is a straightforward  consequence  of Proposition \ref{plower}.  

In the next step, we prove the convexity of set $\mathcal{S}$. Let $h_1, h_2\in \mathcal{S}$. For any $\beta\in [0,1]$, $h=\beta h_1+(1-\beta) h_2$ is still a density function, and the first and the last properties in \eqref{sets} are trivial for $h$. It suffices to show that for any $\beta\in [0,1]$,
\[
Q_{\alpha}(\beta h_1+(1-\beta)h_2)\leq K,
\]
which is true, because
\begin{align*}
&\int_{-\infty}^{-K} dx_j  \int_{\RR^{n-1}}(\beta h_1(x)+(1-\beta)h_2(x))  \prod_{i\ne j, 1\le i \le n}dx_i\\
=&\beta\int_{-\infty}^{-K} dx_j  \int_{\RR^{n-1}} h_1(x) \prod_{i\ne j, 1\le i \le n}dx_i +(1-\beta)\int_{-\infty}^{-K} dx_j  \int_{\RR^{n-1}}h_2(x)  \prod_{i\ne j, 1\le i \le n}dx_i\\
\leq &\beta \alpha_j +(1-\beta)\alpha_j=\alpha_j,
\end{align*}
and
\begin{align*}
\int_K^{\infty} dx_j  \int_{\RR^{n-1}}(\beta h_1(x)+(1-\beta)h_2(x)) \prod_{i\ne j, 1\le i \le n}dx_i\leq 1-\alpha_j,
\end{align*}
for all $j=1,\dots, n$. The proof of this proposition is completed. 
\end{proof} 

The next Feynman-Kac formula for time-inhomogeneous PDE  is cited from    
\cite[page 131-132]{mark1985functional}. Consider the following PDE,
\begin{align}\label{cauchy}
\begin{cases}
&\frac{\partial}{\partial t}u_t(x)= \frac{1}{2} \sum_{i,j=1}^{n} a_{ij} (t,x) \frac{\partial^2}{\partial x_i\partial x_j} u_t(x) + \langle b(t,x), \nabla u_t(x)\rangle+ c(t,x) u_t(x)\,,  \\
&u_0(x) = f(x)\,.
\end{cases}
\end{align}
Let $t>0$ and $x\in \RR^n$  and let  the process $(X_s^{t,x})_{0\leq s\leq t}$ be  the solution to the  following stochastic differential  equation:
\begin{align}\label{fk}
\begin{cases}
dX_s^{t,x} = \sigma\(t-s, X_s^{t,x}\)dW_s + b\(t-s, X_s^{t,x}\)ds,\quad 0\leq s\leq t,\\
 X_0^{t,x} = x\,. 
 \end{cases}
\end{align}
Then, we have the following Feynman-Kac formula for the solution to \eqref{cauchy}. 
\begin{theorem}[see \cite{mark1985functional}]\label{thmfk}
Assume that the entries of the matrix $\sigma(t,x)$ are continuous and bounded on the set   $[0,\infty)\times \RR^n$ and Lipschitz continuous in $x$ with a Lipschitz constant which does not depend on $t$.  The vector $b(t,x)$ are also assumed to be  continuous, 
and Lipschitz continuous in $x$ with a Lipschitz constant which does not depend on $t$.
Let $u$ be the solution to  \eqref{cauchy}. Suppose that $u$ is continuous and bounded on $[0,T] \times \RR^n$.
Suppose further that the time derivative of $u$ and its spatial derivatives up to order two are bounded and continuous in the region $(h<t<T, x \in \RR^{n})$ for every $h \in (0,T)$. Then,  for any $t\in [0,T]$ and $x\in \RR$, $u_t(x)$ can be represented in the form
\begin{align}\label{Feynman}
u_t(x) = \EE f\(X_t^{t,x}\) \exp\( \int_0^t c(t-s, X_s^{t,x})ds\),
\end{align}
where $X_s^{t,x}$ is  given  by \eqref{fk}.
\end{theorem}
\begin{remark}
Note that the Hypotheses (H1)-(H5) imposed on coefficients $F$ and $\sigma$ imply all the conditions required in Theorem \ref{thmfk}. In other words, the density of $X^{\omega}$ of SDE \eqref{auxiliaryeq} can be represented by \eqref{Feynman} under Hypotheses (H1)-(H5).
\end{remark}

\begin{proposition}\label{hypie}
Assume  hypotheses (H1)-(H5). Let 
$u^{\omega}$ be the solution to \eqref{omega_pde} with initial condition $f$ satisfying hypothesis (I). Then, we have  for any $t_0>0$, 
\begin{align}\label{unbd}
U':=\sup_{\substack{t\in[ t_0,T]\\ \omega\in C_b([0,T];\RR^n)}}&\bigg(\int_{\RR^{n} } \sup_{|y|\geq r} u^{\omega}_t(y)^2 \(|r|^{{4n}+\ep-1} +r^{n-1}\)  dr \nonumber\\
&+\int_0^{\infty} \Big[\sup_{|z|\geq \lambda}|\nabla u^{\omega}_t(z)|^4\Big]\(\lambda^{4n-1+\ep} +\lambda^{n-1}\) d\lambda \bigg)<\infty.
\end{align}
\end{proposition}

\begin{proof}
We shall  show that the second term in \eqref{unbd} is uniformly bounded.  The uniform boundedness of the first term can be proved in a similar way.  Using Theorem \ref{thmfk}, for any $t\in [0,T]$ and $x\in\RR$, we can write
\begin{align*}
u^{\omega}_t(x) = \EE \(f(X_{t}^{\omega,t,x})\exp\(\int_0^t c \(t-s, \omega_{t-s}, X_{s}^{\omega,t,x}\) ds \)\),
\end{align*}
where $X^{\omega, t,x}$ is the solution to \eqref{fk},   where $\sigma(t,x)$ and $b(t,x)$ are replaced by  $\sigma(t, \om(t), x)$ and  $F(t, \om(t), x)$. 
Differentiating this expression with respect to  $x$,  we have
\begin{align*}
\nabla u^{\omega}_t(x) = &\EE \bigg[\exp\(\int_0^t c \(t-s, \omega,  X_{s}^{\omega,t,x}\) ds \) \nabla f(X_{t}^{\omega,t,x}) \nabla X_{t}^{\omega,t,x} \\
&+ f(X_{t}^{\omega,t,x}) \exp\(\int_0^t c \(t-s, \omega_{t-s},  X_{s}^{\omega,t,x}\) ds \)\\
&\int_0^t  \nabla c \(t-s, \omega_{t-s},  X_{s}^{\omega,t,x}\)\nabla X_{s}^{\omega,t,x}  ds \bigg].
\end{align*}
Due to hypotheses (H4), we know that $c$, $\nabla c$ are both bounded functions. By Cauchy-Schwarz's and Minkowski's inequalities, we can show that
\begin{align}\label{unbd1}
\big|\nabla u^{\omega}_t(x)\big| 
\le& C \Big[\big\| |\nabla f(X_{t}^{\omega,t,x})| \big\|_2 \big\| |\nabla X_{t}^{\omega,t,x} |\big\|_2 +  \big\|f(X_{t}^{\omega,t,x})\big\|_2  \int_0^t \big\||\nabla X_{s}^{\omega,t,x}|\big\|_2ds.
\end{align}
Note that for any $r\in (0,t)$, $\nabla X_r^{\omega,t,x}$ satisfies the following equation
\begin{align*}
\nabla X_r^{\omega,t,x} =& I_n+ \int_0^r \nabla (D\sigma)\(t-s, \omega_{t-s}, X_s^{\omega,t,x}\)\nabla X_s^{\omega,t,x}  dW_s \\
&+ \int_0^r \nabla F\(t-s, \omega_{t-s},X_s^{\omega,t,x}\) \nabla X_s^{\omega,t,x} ds.
\end{align*}
From Burkholder-Davis-Gundy's  and Jensen's and Minkowski's inequalities it follows 
that
\begin{align*}
\big\| |\nabla X_r^{\omega,t,x}|\big\|^2_2 \le & n+T^{\frac{1}{2}}\int_0^r \big\||\nabla F\(t-s, \omega_{t-s},X_s^{\omega,t,x}\) \nabla X_s^{\omega,t,x}| \big\|^2_2ds\\
&+  \int_0^r \big\|| \nabla (D\sigma)\(t-s, \omega_{t-s}, X_s^{\omega,t,x}\)\nabla X_s^{\omega,t,x}| \big\|^2_2 ds \\
\le &n+ C(\kappa,T) \int_0^r  \big\||\nabla X_s^{\omega,t,x}|\big\|^2_2 ds.
\end{align*}
By Gronwall's inequality, we obtain
\[
\big\||\nabla X_r^{\omega,t,x}|\big\|_2 \le \sqrt{ne^{C(\kappa,T)r}}\le \sqrt{ne^{C(\kappa,T)T}}.
\]
Inserting this inequality into  \eqref{unbd1},   we  obtain 
\begin{align*}
\big|\nabla u^{\omega}_t(x)\big| \le & C(\kappa, T) \big(\big\| |\nabla f(X_{t}^{\omega,t,x})| \big\|_2 + \big\|f(X_{t}^{\omega,t,x})\big\|_2 \big).
\end{align*}
This implies 
\begin{align}\label{deriofut}
&\int_0^{\infty} \Big[\sup_{|z|\geq \lambda}|\nabla u^{\omega}_t(z)|^4\Big]\(\lambda^{4n-1+\ep} +\lambda^{n-1}\) d\lambda\nonumber\\
\leq &C\int_0^{\infty} \Big[\sup_{|z|\geq \lambda}\big(\big\| |\nabla f(X_{t}^{\omega,t,z})| \big\|_2 + \big\|f(X_{t}^{\omega,t,z})\big\|_2 \big)^4\Big]\(\lambda^{4n-1+\ep} +\lambda^{n-1}\) d\lambda\nonumber\\
\leq &C\Bigg(\int_0^{\infty} \bigg[\sup_{|z|\geq \lambda} \Big(\int_{\RR^n}|\nabla f(x)|^2p_t^{\omega}(x,z)dx\Big)^2\bigg]\(\lambda^{4n-1+\ep} +\lambda^{n-1}\) d\lambda\nonumber\\
&\quad+\int_0^{\infty} \bigg[\sup_{|z|\geq \lambda} \Big(\int_{\RR^n}| f(x)|^2p_t^{\omega}(x,z)dx\Big)^2\bigg]\(\lambda^{4n-1+\ep} +\lambda^{n-1}\) d\lambda\Bigg)\nonumber\\
:=&C(D_1+D_2),
\end{align}
where $p^{\omega}_t(\cdot,\xi)$ is the probability density of the solution to \eqref{auxiliaryeq} with initial condition $X^{\omega}_0=\xi\in \RR^n$.  Applying Theorem \ref{tse} and Jensen's inequality, we can show that 
\begin{align}\label{d1}
D_1\leq &C\int_0^{\infty} \bigg[\sup_{|z|\geq \lambda} \int_{\RR^n}|\nabla f(x)|^4\frac{C_T}{t^{n^2/2}}\exp \( -\frac{ |\mathcal{T}_t^{-1}(\theta^{\omega}(t, z) - x)|^2}{C_T}\) dx\bigg]\(\lambda^{4n-1+\ep} +\lambda^{n-1}\) d\lambda.
\end{align}
Now that hypothesis (I) implies that 
\begin{align}\label{nf1}
\int_{\RR^n}|\nabla f(x)|^4dx\leq \int_0^{\infty}\sup_{|x|\geq \lambda}|\nabla f(x)|^4\lambda^{n-1}d\lambda \leq U
\end{align}
and
\begin{align}\label{nf2}
 \sup_{|x|\geq \delta}|\nabla f(x)|\leq \delta^{-1}\int_0^{\delta}\sup_{|x|\geq \lambda}|\nabla f(x)|d\lambda\leq \delta^{-1}U,
\end{align}
for any $\delta>0$.
Using \eqref{nf1}  we obtain 
\begin{align*}
&\int_{\RR^n}|\nabla f(x)|^4\frac{C_T}{t^{n^2/2}}\exp \( -\frac{ |\mathcal{T}_t^{-1}(\theta^{\omega}(t, z) - x)|^2}{C_T}\) dx\\
= &\int_{|x|\leq \delta}|\nabla f(x)|^4\frac{C_T}{t^{n^2/2}}\exp \( -\frac{ |\mathcal{T}_t^{-1}(\theta^{\omega}(t, z) - x)|^2}{C_T}\) dx\\
&+\int_{|x|> \delta}|\nabla f(x)|^4\frac{C_T}{t^{n^2/2}}\exp \( -\frac{ |\mathcal{T}_t^{-1}(\theta^{\omega}(t, z) - x)|^2}{C_T}\) dx\\
\leq &\frac{C_T}{t^{n^2/2}}e^{\frac{(t^{-(2n-1)}\vee t^{-1})\delta^2}{C_T}}U\exp \( -\frac{ |\mathcal{T}_t^{-1}(\theta^{\omega}(t, z) )|^2}{C_T}\)\\
&+C_T\int_{\RR^n}\mathbf{1}_{\left\{ |\theta^{\omega}(t,z)-\mathcal{T}_t(y)|> \delta\right\} }|\nabla f(\theta^{\omega}(t,z)-\mathcal{T}_t(y))|^4\exp \( -\frac{ |y|^2}{C_T}\) dy,
\end{align*}
in the second part of the last inequality we perform change of variable  $x \to y=\mathcal{T}_t^{-1}(\theta^{\omega}(t, z) - x)$.

By Lemma \ref{lmma}, we have  that 
\begin{align*}
|\theta^{\omega}(t, z)|\geq (e^{-\kappa T}|z|-\kappa T)\mathbf{1}_{
\left\{|z|\geq e^{\kappa T}\kappa T\right\} }.
\end{align*}
This implies that 
\begin{align}\label{d11}
&\int_0^{\infty} \frac{C_T}{t^{n^2/2}}e^{\frac{(t^{-(2n-1)}\vee t^{-1})\delta^2}{C_T}}U \bigg[\sup_{|z|\geq \lambda}\exp \( -\frac{ |\mathcal{T}_t^{-1}(\theta^{\omega}(t, z) )|^2}{C_T}\)\bigg]\(\lambda^{4n-1+\ep} +\lambda^{n-1}\) d\lambda\nonumber\\
\leq &  \frac{C_T}{t^{n^2/2}}e^{\frac{(t^{-(2n-1)}\vee t^{-1})\delta^2}{C_T}}U\int_{ e^{\kappa T}\kappa T}^{\infty}\exp \( -\frac{ (t^{-2n+1}\wedge t^{-1})(e^{-\kappa T}\lambda-\kappa T)^2}{C_T}\) \nonumber\\
&\qquad\qquad \(\lambda^{4n-1+\ep} +\lambda^{n-1}\) d\lambda\nonumber\\
&+ \frac{C_T}{t^{n^2/2}}e^{\frac{(t^{-(2n-1)}\vee t^{-1})\delta^2}{C_T}}U\int_0^{e^{\kappa T}\kappa T} \(\lambda^{4n-1+\ep} +\lambda^{n-1}\) d\lambda\leq C,
\end{align}
for some $C$ depending on $n, C_T, t_0,T,\epsilon, \kappa$ and $U$. Similarly, on the set 
\[
\{|z|\geq \lambda \}\cap \{|\theta^{\omega}(t,z)-\mathcal{T}_t(y)|> \delta\},
\]
we can deduce that
\[
|\theta^{\omega}(t,z)-\mathcal{T}_t(y)|\geq \delta \vee (e^{-\kappa T}\lambda-\kappa T-|\mathcal{T}_t(y)|).
\]
Therefore, it follows from \eqref{nf2} that
\begin{align}\label{d12}
&\int_0^{\infty} \bigg[\sup_{|z|\geq \lambda}\int_{\RR^n}\mathbf{1}_{\left\{ |\theta^{\omega}(t,z)-\mathcal{T}_t(y)|> \delta\right\} }|\nabla f(\theta^{\omega}(t,z)-\mathcal{T}_t(y))|^4\exp \( -\frac{ |y|^2}{C_T}\) dy\bigg]\nonumber\\
&\qquad\qquad \(\lambda^{4n-1+\ep} +\lambda^{n-1}\) d\lambda\nonumber\\
\leq &  \int_{\RR^n}\int_0^{\infty} \bigg[\sup_{|\tilde{z}|\geq \delta \vee (e^{-\kappa T}\lambda-\kappa T-|\mathcal{T}_t(y)|)}|\nabla f(\tilde{z})|^4 \bigg]\(\lambda^{4n-1+\ep} +\lambda^{n-1}\) e^{ -\frac{ |y|^2}{C_T}}d\lambda dy\nonumber\\
\leq &\int_{\RR^n}\int_0^{\infty} \mathbf{1}_{\left\{e^{-\kappa T}\lambda-\kappa T-|\mathcal{T}_t(y)|<\delta\right\} }\bigg[\sup_{|\tilde{z}|\geq \delta}|\nabla f(\tilde{z})|^4 \bigg]\(\lambda^{4n-1+\ep} +\lambda^{n-1}\) e^{ -\frac{ |y|^2}{C_T}}d\lambda dy\nonumber\\
&+\int_{\RR^n}\int_0^{\infty} \mathbf{1}_{\left\{ e^{-\kappa T}\lambda-\kappa T-|\mathcal{T}_t(y)|\geq \delta\right\} }\bigg[\sup_{|\tilde{z}|\geq e^{-\kappa T}\lambda-\kappa T-|\mathcal{T}_t(y)|}|\nabla f(\tilde{z})|^4 \bigg]\nonumber\\
&\qquad\qquad  \(\lambda^{4n-1+\ep} +\lambda^{n-1}\) e^{ -\frac{ |y|^2}{C_T}}d\lambda dy\nonumber\\
\leq &\delta^{-1}U\int_{\RR^n}\int_0^{e^{\kappa T}(\delta+\kappa T+|\mathcal{T}_t(y)|) }  \(\lambda^{4n-1+\ep} +\lambda^{n-1}\)d\lambda e^{ -\frac{ |y|^2}{C_T}} dy+e^{(4n+\ep)\kappa T}\int_{\RR^n}\int_{0}^{\infty}e^{ -\frac{ |y|^2}{C_T}}\nonumber\\
&\quad \times \bigg[\sup_{|\tilde{z}|\geq \tau}|\nabla f(\tilde{z})|^4 \bigg]\(|\tau+\kappa T+|\mathcal{T}_t(y)||^{4n-1+\ep} +|\tau+\kappa T+|\mathcal{T}_t(y)||^{n-1}\) d\tau dy\leq C,
\end{align}
where $C>0$ depends on $n, C_T, t_0,T,\epsilon, \kappa$, $U$ and $\delta$.  Combining  \eqref{d1}, \eqref{d11} and \eqref{d12} we see that  $D_1$ is bounded. Using a similar argument, we can   prove that $D_2$ is bounded uniformly in $t\in [t_0, T]$ and $\omega\in C_b([0,T];\RR^n)$. The proof of this proposition is then completed.
\end{proof}
\begin{remark} The main difficulty in the above proof is to show the integrability over an unbounded  domain with respect to $\lambda$. After \eqref{nf2} we divided the integral domain   into $|x|\le \delta$ and $|x|>\delta$ is for simplicity because even if we use 
 $|x|\le \delta \sqrt t$ and $|x|>\delta \sqrt t$, we cannot get rid of the $t_0$ 
 in the statement \eqref{unbd}. 
\end{remark}

\section{Proof of the main results}\label{s.4}

Now, we are ready to  to prove Theorems \ref{SDEsu} and \ref{PDEthm}.  In the first subsection we shall prove the existence and uniqueness of the local solution to PDE \eqref{quan_pde} 
up to a small time $t_0$. 

\subsection{Local solution}
In this subsection, we prove a local version of Theorem \ref{PDEthm} (see Proposition \ref{lcPDEthm}). We shall use the fixed point theorem. 
First we need to bound the distance of quantiles by the distance of  distributions. 
The following lemma is known (see e.g.    \cite{kolokoltsov2013nonlinear}).  We rewrite a short proof   for the sake of completeness.
\begin{lemma}\label{lmmaqt}
Let $\alpha\in (0,1)^n$ and let $K, \delta, \varepsilon$ be positive constants. Denote by $\mathcal{S}$ the collection of density functions satisfying   \eqref{sets}. Then, for any $h_1,h_2\in \mathcal{S}$,
\begin{align}\label{ql1}
|Q_{\alpha}(h_1)-Q_{\alpha}(h_2)|\leq \sqrt{n}(2K)^{-(n-1)}\delta^{-1} |h_1 - h_2 |_{L^1}.
\end{align}
\end{lemma}
\begin{proof}
Since $h_1,h_2\in \mathcal{S}$ where $\mathcal{S}$ is a convex set, we know that  for any $\beta \in (0,1)$,
\[
h^{\beta}:=\beta h_1+(1-\beta)h_2\in \mathcal{S}
\]
as well. Write $\hat{\omega}(\beta)=(\hat{\omega}_1(\beta),\dots,\hat{\omega}_n(\beta))= Q_{\alpha}(h^{\beta})$. 

By definition of the quantile, for any $j = 1, \cdots, n$,  we have 
\begin{align}\label{quane}
\int_{-\infty}^{\hat{\omega}_j(\beta)} dx_j \int_{\RR^{n-1}} h^{\beta}(x) \prod_{k\ne j}d x_k=\alpha_j.
\end{align}
Differentiating  both sides of \eqref{quane} with respect to $\beta$ yields
\[
\hat{\omega}_j'(\beta)\int_{\RR^{n-1}} h^{\beta}(x) \prod_{k\ne j}d x_k\bigg|_{x_j=\hat{\omega}_j(\beta)}+\int_{-\infty}^{\hat{\omega}_j(\beta)} dx_j \int_{\RR^{n-1}} \(h_1(x)-h_2(x)\) \prod_{k\ne j}d x_k=0\,. 
\]
Thus
\[
\hat{\omega}_j'(\beta) = -\bigg[\int_{\RR^{n-1}} h^{\beta}(x) \prod_{k\ne j}d x_k\bigg|_{x_j=\hat{\omega}_j(\beta)}\bigg]^{-1}\int_{-\infty}^{\hat{\omega}_j(\beta)} dx_j \int_{\RR^{n-1}} \(h_1(x)-h_2(x)\) \prod_{k\ne j}d x_k.
\]
It follows that
\begin{align}\label{useprop}
|Q^j_{\alpha}( h_1)  - Q^j_{\alpha}( h_2)| &= |\hat{\omega}_j(1) - \hat{\omega}_j(0)| =\bigg|\int_0^1 \hat{\omega}_j'(\beta) d\beta\bigg| \nonumber\\
&=\bigg| \int_0^1\frac{\int_{-\infty}^{\hat{\omega}_j(\beta)} dx_j \int_{\RR^{n-1}} \(h_1(x)-h_2(x)\) \prod_{k\ne j}d x_k}{\int_{\RR^{n-1}} h^{\beta}(x) \prod_{k\ne j}d x_k\big|_{x_j=\hat{\omega}_j(\beta)}} d\beta\bigg| \nonumber\\
&\le  |h_1 - h_2|_{L^1} \bigg| \int_0^1\bigg[\int_{\RR^{n-1}} h^{\beta}(x) \prod_{k\ne j}d x_k\bigg|_{x_j=\hat{\omega}_j(\beta)}\bigg]^{-1} d\beta\bigg|.
\end{align}
Recall that $h^{\beta}\in \mathcal{S}$. This implies that $\max_{1\leq j\leq n}|\hat{\omega}_j(\beta)|\leq K$, and thus by \eqref{sets}  we have 
\begin{align*}
\int_{\RR^{n-1}} h^{\beta}(x) \prod_{k\ne j}d x_k\bigg|_{x_j=\hat{\omega}_j(\beta)}\geq &\int_{[-K,K]^{n-1}} h^{\beta}(x) \prod_{k\ne j}d x_k\bigg|_{x_j=\hat{\omega}_j(\beta)}\\
\geq& \int_{[-K,K]^{n-1}} \delta\prod_{k\ne j}d x_k\geq (2K)^{n-1}\delta.
\end{align*}
As a consequence, we have 
\[
|Q^j_{\alpha}(h_1)  - Q^j_{\alpha}(h_2)| \le  (2K)^{-(n-1)}\delta^{-1} |h_1-h_2 |_{L^1},
\]
for all $j=1,\dots, n$, which yields the lemma.
\end{proof}

The next proposition describes the dependence of the solution of 
 \eqref{omega_pde} with respect to $\omega$.  It will be used  to bound the distance of 
 distributions  of the solutions to \eqref{simeq} by the quantiles.

\begin{proposition}\label{sensiofu}
	Let the hypotheses (H1)-(H5) be satisfied. Let $u^{(1)}=u^{\omega^{(1)}}$ and $u^{(2)}=u^{\omega^{(2)}}$ be the solutions to equation \eqref{omega_pde}  corresponding to  the continuous functions $\omega=\omega^{(1)}$ and $\omega=\omega^{(2)}$ respectively and with the same initial condition $f$ satisfying hypothesis (I). 
	 Then, the following inequality holds true 
	\begin{align}\label{sen1}
\sup_{s \in [0,t]} |u_s^{(1)} - u_s^{(2)} |_{L^1} & \le C_0\(t+ \sqrt{t}\)\sup_{s \in [0,t]}|\omega_s^{(1)} - \omega_s^{(2)}|,\quad \forall \ t\in [0, T]\,, 
	\end{align}
where $C_0 $ is a positive constant independent of $\omega^{(1)}$, $\omega^{(2)}$ and $t$.
\end{proposition}
\begin{proof}
By the Feynman-Kac formula (Theorem \ref{thmfk}), for $i=1$ and $2$, we can write 
\begin{align*}
u^{(i)}_t(x) = \EE \(f(X_{t}^{(i),t,x})\exp\(\int_0^t c^{(i)} \(t- s, X_{s}^{(i),t,x}\) ds \)\),
\end{align*}
where  $X^{(i),t,x}=X^{\omega^{(i)},t,x}$ is the solution to \eqref{fk} with initial condition $X_0^{(i),t,x}=x$ and coefficients 
\begin{eqnarray*}
&&a^{(i)}(t-s,x)=a(t-s,\omega^{(i)}_{t-s},x),\ b^{(i)}(t-s,x)=b(t-s,\omega^{(i)}_{t-s},x),\\ &&c^{(i)}(t-s,x)=c(t-s,\omega^{(i)}_{t-s},x), 
\end{eqnarray*}
for all $t\in[0,T]$ and $x\in \RR^n$ with $a,b$ and $c$  being defined by
 \eqref{terma}-\eqref{termc} respectively. Thus, we have
\begin{align}\label{e3}
&\quad \int_{\RR^n}| u^{(1)}_t(x) -  u^{(2)}_t(x)|dx  \nonumber \\
=& \int_{\RR^n
} \EE \bigg[ f(X_{t}^{(1),t,x})\exp\(\int_0^tc^{(1)}\(t-s, X_{s}^{(1),t,x}\) ds \)  \nonumber\\
& - f(X_{t}^{(2),t,x})\exp\(\int_0^t c^{(2)}\(t-s, X_{s}^{(2),t,x}\)  ds \)  \bigg]dx  \nonumber \\
=& \int_{\RR^n
} \EE \bigg\{ f(X_{t}^{(1),t,x})\left[\exp\(\int_0^tc^{(1)}\(t-s, X_{s}^{(1),t,x}\) ds \) -\exp\(\int_0^t c^{(2)}\(t-s, X_{s}^{(2),t,x}\)  ds \)\right]  \bigg\}dx  \nonumber \\
&\quad+\int_{\RR^n
} \EE \bigg[ \(f(X_{t}^{(1),t,x}) -   f(X_{t}^{(2),t,x})\)\exp\(\int_0^t c^{(2)}\(t-s, X_{s}^{(2),t,x}\) ds \)  \bigg]dx  \nonumber \\
= &I_1+ I_2.
\end{align}
Due to hypothesis  (H4), we know that  the function $c$ is uniformly bounded on $[0,T]\times \RR^{n}\times \RR^n$ by $2\kappa$, and Lipschitz continuous. Then, the first term of \eqref{e3} is bounded by using the mean value theorem as follows: 
\begin{align}\label{e2}
I_1 =& \int_{\RR^n} \EE \bigg[ f(X_t^{(1),t,x}) \bigg(\exp\(\int_0^t c^{(1)}\(t-s, X_{s}^{(1),t,x}\) ds \)  \nonumber\\
&- \exp\(\int_0^t c^{(2)}\(t-s, X_{s}^{(2),t,x}\)  ds \)  \bigg) \bigg] dx  \nonumber\\
\le&  e^{2\kappa T} \int_{\RR^n} \EE \bigg[ f(X_t^{(1),t,x})  \( \int_0^t c^{(1)}\(t-s, X_{s}^{(1),t,x}\) ds -  \int_0^t c^{(2)}\(t-s, X_{s}^{(2),t,x}\)  ds \)  \bigg] dx  \nonumber\\
=&e^{2\kappa T}  \int_{\RR^n} \EE \bigg[ f(X_t^{(1),t,x})  \( \int_0^tc \(t-s,\omega^{(1)}_{t-s}, X_{s}^{(1),t,x}\) -   c \(t-s,\omega^{(1)}_{t-s}, X_{s}^{(2),t,x}\) ds \)  \bigg] dx  \nonumber\\
&+ e^{2\kappa T}  \int_{\RR^n} \EE \bigg[ f(X_t^{(1),t,x})  \( \int_0^tc \(t-s,\omega^{(1)}_{t-s}, X_{s}^{(2),t,x}\) -   c \(t-s,\omega^{(2)}_{t-s}, X_{s}^{(2),t,x}\) ds \)  \bigg] dx  \nonumber\\
\le& c_{\kappa,T} \int_0^t |\omega^{(1)}_{t-s} - \omega^{(2)}_{t-s}|ds\int_{\RR^n}  \EE [f(X_t^{(1),t,x})] dx \nonumber\\
& +  c_{\kappa,T} \int_{\RR^n}  \EE \bigg[f(X_t^{(1),t,x}) \int_0^t |X_{s}^{(1),t,x} - X_{s}^{(2),t,x}|ds \bigg] dx   \nonumber \\
=& c_{\kappa,T}  \(I_{11} + I_{12}\),
\end{align}
where $c_{\kappa,T}$ is a positive constant depending on $\kappa$ and $T$. For $i=1,2$, denote by $p^{(i)}_t(\cdot,x)$ the probability density of $X^{(i)}_t$ and write $\theta^{(i)}=\theta^{\omega^{(i)}}$ the solution to \eqref{determinF} with $\omega=\omega^{(i)}$.  Then, by Theorem \ref{tse} and Corollary \ref{coro}, we have 
\begin{align}\label{intp}
\int_{\RR^n}\EE [f(X_t^{(1),t,x})] dx&=\int_{\RR^{2n}}  f(y) p^{(1)}_t(y,x) dy dx \nonumber\\
&\le C_T\int_{\RR^{2n}}  f(y)  t^{-n^2/2} \exp \( - C_T^{-1} \sum_{i=1}^{n}\(\frac{\theta^{(1)}_i(t, x) - y_i}{t^{i-\frac{1}{2}}}\)^2\) dxdy \nonumber \\
& \le C_T\int_{\RR^{2n}}  f(y) \exp \( - C_T^{-1} |z|^2\)  \det \(\nabla (\theta^{(1)})^{-1}\(t, y-\mathcal{T}_t(z)\)\)dzdy \nonumber\\
&\leq  C_Te^{n\kappa T}\int_{\RR^n}  f(y)dy\int_{\RR^n} \exp \( - C_T^{-1} |z|^2\) dz \nonumber\\
& \le  C_Te^{n\kappa T} (C_T\pi)^{\frac{n}{2}}.
\end{align}
Hence,
\begin{align}\label{I11}
I_{11} \le&  C_1t\sup_{s \in [0,t]} |\omega_s^{(1)} - \omega_s^{(2)}|,
\end{align}
for some positive constant $C_1$ independent of $\omega^{(1)}$, $\omega^{(2)}$ and $t$. On the other hand, for any $p\geq 1$, we can deduce that, for some constant $c_{n,p}>0$ depending on $n$ and $p$,
\begin{align}\label{e1}
&\quad \EE \big| X_{t}^{(1),t,x} - X_{t}^{(2),t,x} \big|^{2p}  \nonumber\\
& \le c_{n,p}\bigg[\sum_{i=1}^{n} \EE \(\int_0^t\(F_i(t-s,\omega_{t-s}^{(1)},X^{(1),t,x}_s)  - F_i(t-s,\omega_{t-s}^{(2)}, X^{(2),t,x}_s)\)ds\)^{2p} \nonumber\\
&\quad+ \EE \(\int_0^t \(\sigma(t-s,\omega_{t-s}^{(1)}, X^{(1),t,x}_s)  - \sigma(t-s,\omega_{t-s}^{(2)}, X^{(2),t,x}_s) \)dW_s\)^{2p} \bigg].
\end{align}
By hypothesis  (H1) and the Burkholder-Davis-Gundy and Jensen's inequalities, we have
\begin{align*}
&\quad \EE \big| X_{t}^{(1),t,x} - X_{t}^{(2),t,x} \big|^{2p} \\
& \le c_{n,p}\kappa\bigg[ t^{2p-1}\(\int_0^t |\omega_{t-s}^{(1)} - \omega_{t-s}^{(2)}|^{2p}ds + \int_0^t \EE |X^{(1),t,x}_s -  X^{(2),t,x}_s|^{2p}ds \)\\
&\quad + \EE \(\int_0^t \(|\omega_{t-s}^{(1)} - \omega_{t-s}^{(2)}|+ |X^{(1),t,x}_s -  X^{(2),t,x}_s| \)^2ds\)^p \bigg] \\ 
&\le c_{n,p,\kappa}\bigg[ t^{2p-1}\(\int_0^t |\omega_{t-s}^{(1)} - \omega_{t-s}^{(2)}|^{2p}ds + \int_0^t \EE |X^{(1),t,x}_s -  X^{(2),t,x}_s|^{2p}ds\) \\
&\quad + t^{p-1}\int_0^t |\omega_{t-s}^{(1)} - \omega_{t-s}^{(2)}|^{2p}ds+ t^{p-1}\int_0^t \EE|X^{(1),t,x}_s -  X^{(2),t,x}_s|^{2p} ds \bigg]\\
& \le  c_{n,p,\kappa} (t^p+t^{2p})\sup_{s \in [0,t]} |\omega_{t-s}^{(1)} - \omega_{t-s}^{(2)}|^{2p}+  c_{n,p,\kappa} (t^{2p-1}+t^{p-1})\int_0^t \EE | X_{s}^{(1),t,x} - X_{s}^{(2),t,x}|^{2p} ds.
\end{align*}
An application of Gronwall's inequality yields  that
\begin{align}\label{estX}
\EE\big| X_{t}^{(1),t,x} - X_{t}^{(2),t,x} \big|^{2p}  \le
 c_{n,p,\kappa}(t^p+t^{2p}) e^{c_{n,p,\kappa} (T^{2p}+T^{p})}\sup_{s \in [0,t]}|\omega_s^{(1)} - \omega_s^{(2)}|^{2p}.
\end{align}
By Fubini's theorem,  H\"{o}lder's inequality and \eqref{estX}, we get that
\begin{align}\label{t1}
I_{12}& = \int_0^t\int_{\RR^n}  \EE \big[f(X_t^{(1),t,x})  |X_{s}^{(1),t,x} - X_{s}^{(2),t,x}| \big] dxds \nonumber\\
& \le  \int_0^t\int_{\RR^n}  \| f(X_t^{(1),t,x})\|_2 \big\| |X_{s}^{(1),t,x} - X_{s}^{(2),t,x}|\big\|_2 dxds\nonumber\\
&\le  c_{n,p,T}\sup_{s \in [0,t]}|\omega_s^{(1)} - \omega_s^{(2)}|\int_0^t\int_{\RR^n}  \| f(X_t^{(1),t,x})\|_2 dxds,
\end{align}
for some positive constant $c_{n,p,T}$ depending on $n,p$ and $T$. Notice that by Theorem \ref{tse}, Corollary \ref{coro} and Cauchy-Schwarz's inequality, we can deduce that 
\begin{align*}
\int_{\RR^n} \| &f(X_t^{(1),t,x}) \|_2 dx =\int_{\RR^n} \(\int_{\RR^n} f(y)^2 p^{(1)}_t(y,x)dy\)^{\frac{1}{2}}dx\nonumber\\
\leq&\frac{\sqrt{C_T}}{t^{n^2/4}}\int_{\RR^n} \(\int_{\RR^n} |f(y)|^2 \exp \( C_T^{-1} |\mathcal{T}_t^{-1}(\theta^{(1)}(t, x) -  y)|^2\) dy \)^{\frac{1}{2}}dx\nonumber\\ 
\le& \frac{\sqrt{C_T}}{t^{n^2/4}}\(\int_{\RR^{2n}}   |f(y)|^2 \exp \( C_T^{-1} |\mathcal{T}_t^{-1}(\theta^{(1)}(t, x) -  y)|^2\)\(|\theta^{(1)}(t, x)|^{\frac{n+\ep}{2}}\vee 1\)^2dydx\)^{\frac{1}{2}}  \nonumber \\
& \times \(\int_{\RR^n} \(|\theta^{(1)}(t, h)|^{\frac{n+\ep}{2}}\vee 1\)^{-2}dh\)^{\frac{1}{2}}.
\end{align*}
By changing of variables $x\to z=\mathcal{T}_t^{-1}(\theta^{(1)}(t, x) -  y)$ and $h\to l = \theta^{(1)}(t, h)$, we can write
\begin{align}\label{r1com1}
&\int_{\RR^n} \| f(X_t^{(1),t,x}) \|_2 dx\nonumber\\
\le & \sqrt{C_T}\bigg[\int_{\RR^{2n}} \det \(\nabla \(\theta^{(1)}\)^{-1}\(t,y+\mathcal{T}_t(z)\)\)  |f(y)|^2 e^{ -\frac{ |z|^2}{C_T}}\(|\mathcal{T}_tz+y|^{n+\ep}\vee 1\)dzdy\bigg]^{\frac{1}{2}}\nonumber\\
&\times\left[\int_{\RR^n} \det \(\nabla\( \theta^{(1)}\)^{-1}\(t,l\)\) \(|l|^{-(n+\ep)}\vee 1\)dl\right]^{\frac{1}{2}} \nonumber\\
\le& c_{n,\epsilon,T}\sqrt{C_T}e^{n\kappa T}\(\int_{\RR^{2n}}   f(y)^2  \exp \( -C_T^{-1} |z|^2\) \(|z|^{n+\ep}+|y|^{n+\ep}+ 1\)  dzdy\)^{\frac{1}{2}}\nonumber\\
&\times\left[\int_{\RR^n}  \(|l|^{-(n+\ep)}\vee 1\)dl\right]^{\frac{1}{2}}.
\end{align}
Recall that $f>0$ is a probability density satisfying hypothesis (I).  \eqref{r1com1} tells us that 
\begin{align}\label{r1com}
\int_{\RR^n} &\| f(X_t^{(1),t,x}) \|_2 dx\leq C,\quad \forall t\in [0,T],
\end{align}
where $C>0$ depends on $C_T, n, p, \kappa,\epsilon, T$ and $U$. Combining inequalities \eqref{t1} and \eqref{r1com}, we finally obtain
\begin{align}\label{I12}
I_{12} \le C_1t\sup_{s \in [0,t]}|\omega_s^{(1)} - \omega_s^{(2)}|,
\end{align}
for some $C_1$ independent of $\omega^{(1)}, \omega^{(2)}$ and $t$. 

In the next step, we estimate  the  term $I_2$ in \eqref{e3}. By Cauchy-Schwarz's inequality and the fact that $c$ is uniformly bounded, we can write
\begin{align}\label{inofI2}
I_2 &\le  \int_{\RR^n} \|f(X_{t}^{(1),t,x}) - f(X_{t}^{(2),t,x})\|_2 \bigg\|\exp\(\int_0^tc^{(2)}\(t-s, X_{s}^{(2),t,x}\) ds\) \bigg\|_2 dx \nonumber\\
&\le e^{2\kappa T} \int_{\RR^n} \|f(X_{t}^{(1),t,x}) - f(X_{t}^{(2),t,x})\|_2 dx.
\end{align}
To bound the above integral, we first claim  the following 
 version of mean value theorem. 	For any $x,y\in \RR^n$, the following inequality holds true: 
	\begin{align}\label{fxy}
	|f(x)-f(y)|\leq 2 \sup_{|x|\wedge |y|\leq |z|\leq |x|\vee |y|}|\nabla f(z)||x-y|.
	\end{align}
In fact,  consider   a plane $\mathcal{P}$ such that $0, x, y\in \mathcal{P}$. Without loss of  generality, suppose that $|x|\leq |y|$. Let $x'$ be the intersection of the straight line connecting $0$ and $y$, and the circle $\mathcal{O}$ centered at $0$ with radius $|x|$. Applying the fundamental theorem of calculus  to the path integral of $\nabla f$ along the (shorter) arc $x\to x'$ on $\mathcal{O}$, and then along the straight line $x'\to y$, 
we obtain  immediately,
	\begin{align} \label{path}
	|f(x)-f(y)|\leq \sup_{|x|\leq |z|\leq |y|}|\nabla f(z)|(|\wideparen{xx'}|+|y-x'|) 
	\end{align}
where $\wideparen{xx'}$ denotes the arc length.  Since the angle between the ray 
$x'y$ and the line $x'x$ is greater than or equal to $\pi/2$, we see that both
$\wideparen{xx'}$ and $|y-x|$ are less than or equal to $|y-x|$. Thus, 
inequality \eqref{fxy} follows immediately  from \eqref{path}. 
It is worth noticing that we do not apply the mean value theorem on the straight line $x\to y$. Since if so, we have $|f(x)-f(y)|\leq |\nabla f(\xi)| |x-y|$, where the point $\xi
=t_0x+(1-t_0)y$  for some $t_0\in [0, 1]$.  We can have $|\xi|\le |x|\vee |y|$. However,
we cannot guarantee $|\xi|\geq |x|\wedge |y|$, which is critical in the following 
immediate application.

	Using   \eqref{fxy} and Cauchy-Schwarz's inequality, we can write
	\begin{align}\label{I2term1}
	\|f(X_{t}^{(1),t,x}) - f(X_{t}^{(2),t,x})\|_2 \leq \big\|g\big(|X^{(1),t,x}_t|\wedge |X^{(2),t,x}_t|\big)\big\|_4\big\|X^{(1),t,x}_t-X^{(2),t,x}_t\big\|_4,
	\end{align}
	where $g:\RR_+\to \RR$ is given by
	\[
	g(\lambda) :=\sup \{|\nabla f(z)|: |z|\geq \lambda\},\quad \forall \lambda\geq 0.
	\]
Notice that $g(\lambda_1\wedge \lambda_2)\leq g(\lambda_1)+g(\lambda_2)$ for all $\lambda_1,\lambda_2\geq 0$. It follows that  
\begin{align}
	&\int_{\RR^n}\big\|g\big(X^{(1),t,x}_t\wedge X^{(2),t,x}_t\big)\big\|_4dx\leq \int_{\RR^n}\big\|g\big(X^{(1),t,x}_t \big)\big\|_4dx+\int_{\RR^n}\big\|g\big(X^{(2),t,x}_t \big)\big\|_4dx.
\end{align}
Therefore, proceeding  with a similar argument to  that in \eqref{r1com1} and \eqref{r1com} and recalling hypothesis (I), we   deduce that
\begin{align}\label{g12}
&\int_{\RR^n} \| g(X_t^{(1),t,x}) \|_4 dx\nonumber\\
\le& \frac{\sqrt{C_T}}{t^{n^2/4}}\(\int_{\RR^{2n}}   |g(|y|)|^4 \exp \( C_T^{-1} |\mathcal{T}_t^{-1}(\theta^{(1)}(t, x) -  y)|^2\)\(|\theta^{(1)}(t, x)|^{\frac{3(n+\ep)}{4}}\vee 1\)^4dydx\)^{\frac{1}{4}}  \nonumber \\
& \times \(\int_{\RR^n} \(|\theta^{(1)}(t, x)|^{\frac{3(n+\ep)}{4}}\vee 1\)^{-\frac{4}{3}}dx\)^{\frac{3}{4}} \nonumber\\
\le& c_{n,\epsilon}\sqrt{C_T}e^{n\kappa T}\(\int_{\RR^{2n}}   |g(|y|)|^4  \exp \( -C_T^{-1} |z|^2\) \(|z|^{3(n+\ep)}+|y|^{3(n+\ep)}+ 1\)  dzdy\)^{\frac{1}{2}}\nonumber\\
&\times\left[\int_{\RR^n}  \(|z|^{-(n+\ep)}\wedge 1\) dz\right]^{\frac{1}{2}}\leq C,
\end{align}
for some constant $C>0$ depends on $C_T, n, p, \kappa,\epsilon$ and $U$. Combining  inequalities \eqref{estX}, \eqref{inofI2}, \eqref{I2term1} - \eqref{g12}, we get
\begin{align}\label{I2}
I_2 &\le C_2\sqrt{t+t^2}\sup_{s \in [0,t]}|\omega_s^{(1)} - \omega_s^{(2)}|.
\end{align}
Therefore, inequality \eqref{sen1} follows by inserting  inequalities \eqref{I11}, \eqref{I12} and \eqref{I2} into \eqref{e3}.
\end{proof}
\begin{remark}
In formulation \eqref{inofI2}, the function $c^{(2)}(t-s, X_s^{(2),t,x})$ is bounded because of the hypothesis  (H4). This means that the integrability in $x$ has to be guaranteed by that of the term $\|f(X_{t}^{(1),t,x}) - f(X_{t}^{(2),t,x})\|_2$. This is the reason that we assume  the integrability   hypothesis (I)  on $\nabla f$.
\end{remark}
%
%

\begin{proposition}\label{lcPDEthm}
	Assume the  conditions  in Theorem \ref{PDEthm}. Then, there exists $t_0>0$  such that  \eqref{quan_pde} with initial condition $f$ has a   unique solution on the interval $[0, t_0]$.
\end{proposition}
\begin{proof} For any continuous function 
$\omega\in C([0,T], \RR^n)$ by a similar argument to   that in Proposition \ref{sensiofu}, we have that 
\[
\lim_{s\to t} |u^{\omega}_t-u^{\omega}_s|_{L_1}=0.
\]
Then, it follow from \eqref{ql1} that 
\[
\lim_{s\to t}|Q_{\alpha}(u^{\omega}_t)-Q_{\alpha}(u^{\omega}_s)|\leq \lim_{s\to t}\sqrt{n}K^{1-n}\delta^{-1} |u^{\omega}_t - u^{\omega}_s |_{L^1}=0.
\]
In other words, $Q_{\alpha}(u^{\omega}_t)$ is a continuous function in $t$.
  
 We shall use the fix point theorem to prove the proposition.
 Fix a $t_0>0$ satisfying the condition   given by 
\eqref{e.4.25} below.    Let  
$C([0, t_0], \RR^n)$ be the Banach space of all continuous functions with the sup norm. For any $\omega\in C([0, t_0], \RR^n)$, let $u^{\omega}: [0, t_0]\times \RR^n$ be the (unique) solution to \eqref{omega_pde} associated with $\omega$. 
Define 
\begin{align}\label{Banach}
\BB=\left\{ (\om, u^{\omega}), \om \in  C([0, t_0], \RR^n)\right\}
\subseteq  C([0, t_0], \RR^n)\oplus C([0, t_0], L_1(\RR^d))  
\end{align}
with the norm
\begin{align}\label{norm}
\|(\om, u^{\omega} )\|_\BB=\sup_{0\le t\le t_0} |\om(t)|+\sup_{0\le t\le t_0}  
|u_t^{\omega}|_{L_1}\,. 
\end{align}
We claim that 
$\BB$ is a closed set of $C([0, t_0], \RR^n)\oplus C([0, t_0], L_1(\RR^d))$. 
In fact,  if $(\om^{(n)}, u^{\omega^{(n)}})\in \BB$ converges to 
$(\om, v)\in  C([0, t_0], \RR^n)\oplus C([0, t_0], L_1(\RR^d)) $, 
then $\om^{(n)}\rightarrow \om$ in $C([0, t_0], \RR^n)$ and 
$u^{\omega^{(n)}}\rightarrow v$ in $C([0, t_0], L_1(\RR^d))$.  
Thus, $\om \in C([0, t_0], \RR^n)$. Solving   \eqref{omega_pde} associated with $\omega$, we  obtain $u^{\omega} \in C([0, t_0], L_1(\RR^d))$.  
By \eqref{sen1}, we know that  $ u^{\omega^{(n)}}\rightarrow 
u^{\omega}$ in $ C([0, t_0], L_1(\RR^d))$.  This implies that $v=u^{\omega}$.  
In other word, $\BB$ is   closed  and hence it is also a Banach space.

Fix $\alpha=(\alpha_1,\cdots,\alpha_n) \in \RR^n$. Let $K,\delta,\varepsilon>0$ be  defined in \eqref{con1}-\eqref{con3}.  
Now, we define a mapping $\cM: \BB\rightarrow \BB$ as follows
\begin{equation}
\cM(\om, u^{\omega})=(\cM_1(\om, u^{\omega}), \cM_2(\om, u^{\omega}))
\,, \label{e.4.22} 
\end{equation}
where  $(\om, u^{\omega})\in \BB$   and  
\[
\begin{cases}
\cM_1(\om, u^{\omega})=Q_\al(u_\cdot^{\omega})\,,   \\
\cM_2(\om, u^{\omega})=u^{Q_\al(u_\cdot^{\omega})}\,. 
\end{cases}
\] 
Let $\omega^{(1)}$ and $\omega^{(2)}$ be continuous functions on $[0,t_0]$ with values in $\RR^n$, and let $u^{(1)}$ and $u^{(2)}$ be the solutions  to equation \eqref{simeq}  
associated with  $\omega=\omega^{(1)}$ and $\omega=\omega^{(2)}$ respectively,
and with the same initial condition $f$.
Lemma \ref{lmmaqt} and Proposition \ref{sensiofu} imply  that
\begin{align}
\sup_{0\le t \le t_0} |Q_{\alpha}(u_t^{\omega^{(1)}})  - Q_{\alpha}(u_t^{\omega^{(2)}})|\le &C_0\sqrt{n}(2K)^{1-n}\delta^{-1} \(t_0+ \sqrt{t_0}\)\sup_{t\in[0,t_0]}|\omega_t^{(1)} - \omega_t^{(2)}| \label{e.4.23} 
\end{align} 
and 
\begin{eqnarray}
\sup_{0\le t \le t_0}
 |u_t^{Q_\al(u_\cdot^{\omega^{(1)}})}-u_t^{Q_\al(u_\cdot^{\omega^{(2)}})}|_{L_1}
 &\le& C_0\(t_0+ \sqrt{t_0}\)\sup_{s \in [0,t_0]}|Q_\al(u_s^{\omega^{(1)}}) - Q_\al(u_s^{\omega^{(2)}})|\nonumber\\
 &\le& C_0\sqrt{n}(2K)^{1-n}\delta^{-1}  \(t_0+ \sqrt{t_0}\)
 \sup_{t \in [0,t_0]} | u_t^{\omega^{(1)}} - u_t^{\omega^{(2)}} |_{L_1} \,. \nonumber\\
 \label{e.4.24} 
\end{eqnarray}
Choose $t_0>0$ such that 
\begin{equation}
C_0\sqrt{n}(2K)^{1-n}\delta^{-1} \(t_0+ \sqrt{t_0}\)=L < 1\,. 
\label{e.4.25} 
\end{equation}  
 
Then,  from \eqref{e.4.23}-\eqref{e.4.24} 
it follows that the  mapping $\mathcal{M}$ defined by \eqref{e.4.22} 
 is a contraction map on  $\BB$.
It has then a fixed point 
$(\om, u^{\omega})\in \BB$.  By our construction, we see that 
$u^{\omega}$ satisfies \eqref{omega_pde} with 
$\om=Q_{\alpha}(u_t^{\omega })$.  This means that 
$u=u^{\omega}$ satisfies  \eqref{quan_pde}. 


To show the uniqueness, we assume   $v$  is another 
  solution to \eqref{quan_pde}. Letting $\omega' =Q_{\alpha}(v)=\{Q_{\alpha}(v_s)|s\in [0,t_0]\}$, replacing  $Q_{\alpha}(v)$  by $\omega'$ in \eqref{quan_pde}, we see that $v$ is also a solution of \eqref{omega_pde} with $\omega'$. Thus, 
  $(\om' , v)$ is a fixed point of $\cM$. 
By  the uniqueness of the fixed point of map $\cM$, we complete the proof of 
 the 
proposition. 
\end{proof} 

\subsection{Global solution and proof of main result}
In the previous subsection, we proved that   \eqref{quan_pde} has a unique solution $u$ on $[0,t_0]$ when  $t_0$ is small enough. A natural question is  whether this solution can be uniquely extended to any time interval.
 A positive answer is given in this subsection by using Proposition \ref{hypie}.
\begin{proof}[Proof of Theorem \ref{PDEthm}]
By Proposition \ref{lcPDEthm}, there exists $t_0$, such that \eqref{quan_pde} has a unique solution on $[0,t_0]$. Consider $\eqref{quan_pde}$ with $t\ge t_0$ and with 
initial condition $f=u_{t_0}$. Proposition \ref{hypie} can be applied to obtain 
 that there exists $t_1>0$ depending on the initial condition $f=u_{t_0}$ only 
 through 
 $U'$ given  by  \eqref{unbd}  such that 
  equation \eqref{quan_pde} has a unique solution on $[t_0,t_0+t_1]$. Notice that $U'$ is independent of $t\in [t_0,T]$. This allows us to extend the solution of \eqref{quan_pde} repeatedly to the interval $[0, t_0+nt_1]$ until time $t_0+nt_1\ge T$. In other words, \eqref{quan_pde} has a unique solution on the whole  time interval $[0,T]$.
\end{proof}
\begin{proof}[Proof of Theorem \ref{SDEsu}]
Under the hypotheses (H1)-(H5) and (I), Theorem \ref{PDEthm} implies the weak existence and uniqueness to SDE \eqref{simeq}.  Beacause of the weak uniqueness, the $\alpha$-quantile of any weak solution to SDE \eqref{simeq} is the same function on $[0,T]$. Therefore, the strong existence and uniqueness is a straightforward result of Theorem \ref{strongsol}.
\end{proof}


\end{document}